\documentclass{amsart}

\usepackage [T1]{fontenc}
\usepackage[cmtip,all]{xy}
\usepackage{mathrsfs}
\usepackage{mathtools}

\newtheorem{theorem}{Theorem}[section]

\theoremstyle{definition}

\theoremstyle{remark}
\newtheorem{remark}[theorem]{Remark}

\theoremstyle{proposition}
\newtheorem{proposition}[theorem]{Proposition}

\numberwithin{equation}{section}

\newcommand{\F}{\mathbb{F}}
\newcommand{\Q}{\mathbb{Q}}
\newcommand{\Z}{\mathbb{Z}}
\newcommand{\C}{\mathbb{C}}

\begin{document}

\title[isogenies and discrete logarithm problem]{Translating the discrete logarithm problem on Jacobians of genus 3 hyperelliptic curves with $(\ell,\ell,\ell)$-isogenies}

\author{Song Tian}
\address{State Key Laboratory of Information Security, Institute of Information Engineering, Chinese Academy of Sciences, Beijing, 100093, China}
\address{State Key Laboratory of Cryptology, P.O. Box 5159, Beijing, 100878, China}
\email{tiansong@iie.ac.cn}
\thanks{Song Tian is supported by the China Scholarship Council and by the National Natural Science Foundation of China under Grant No. 61802401.}
%\thanks{Song Tian is supported by the China Scholarship Council and by NSFC under Grant No. 61802401.}

\begin{abstract}
We give an algorithm to compute $(\ell,\ell,\ell)$-isogenies from the Jacobians of genus three hyperelliptic curves to the Jacobians of non-hyperelliptic curves.
An important application is to reduce the discrete logarithm problem in the Jacobian of a hyperelliptic curve to the corresponding problem in the Jacobian of a non-hyperelliptic curve.
\end{abstract}

\maketitle

\section{Introduction}
Let $\mathcal{C}$ be a nonsingular curve of genus $g$ over a finite field $\F_q$ and let $V$ be an $\F_q$-rational maximal isotropic subgroup of the $\ell$-torsion of its Jacobian $J_\mathcal{C}$ for an odd prime $\ell\ne \text{char}(\F_q)$. Then the quotient of $J_\mathcal{C}$ by $V$, denoted by $J_\mathcal{C}/V$, is a principally polarized abelian variety over $\F_q$. It is generically isomorphic (over $\F_{q^2}$) to the Jacobian of another curve $\mathcal{D}$ when $g$ is $2$ or $3$ \cite{Ritzenthaler_jacobian,Oort}. In this work we consider the case where $g$ equals $3$, $\mathcal{C}$ is hyperelliptic, while $\mathcal{D}$ is non-hyperelliptic. Our aim is to translate the discrete logarithm problem (DLP) from $J_\mathcal{C}$ to $J_\mathcal{D}$ via the quotient map $\psi:J_\mathcal{C}\to J_\mathcal{D}$ which is called an $(\ell,\ell,\ell)$-isogeny.
This is motivated by the fact that, as Smith \cite{smith222} first realized, explicit isogenies from $J_\mathcal{C}$ to $J_\mathcal{D}$ can make the discrete logarithm weak, since
the DLP in $J_\mathcal{C}$ can be solved in an expected time $\tilde{O}(q^{4/3})$ by using a double large prime variant of index calculus method \cite{double_lp}, while the DLP in $J_\mathcal{D}$ can be solved in time $\tilde{O}(q)$ \cite{diem_nonhyperelliptic}.

Isogeny computation for the genus one case has been solved by V\'elu \cite{velu}. He finds bases of vector spaces of rational functions which are invariant under translations by elements of given finite subgroups. So one can
explicitly write down equations for the resulting elliptic curve and isogeny.
The genus two case has been studied by Dolgachev and Lehavi \cite{Dolgachev}, and Smith \cite{Smith_g2}, who give an explicit geometrical interpretation and an efficient algorithm for $\ell=3$. For $g\ge2$, Lubicz and Robert \cite{CIBAV} provide a general method for computing $(\ell,\ldots,\ell)$-isogenies between principally polarized abelian varieties of dimension $g$ described in terms of algebraic theta coordinates. Their method has complexity $\tilde{O}(\ell^{rg/2})$, where $r=2$ if $\ell$ is a sum of two squares and $r=4$ otherwise.
Using this method, together with formulae for the conversion between Mumford representation and theta coordinates, Cosset and Robert \cite{Cosset-Robert} treat the case $g=2$.

Another interesting treatment is introduced by Couveignes and Ezome \cite{Couveignes-Ezome}.
They first describe how to construct and evaluate functions on the quotient $J_\mathcal{C}/V$ in general. Then they focus on the case $g=2$,
and compute equations for the Kummer variety of $J_\mathcal{C}/V$ by evaluating functions at enough points. Using the geometry of Kummer surfaces, they find an equation for $\mathcal{D}$.
Finally, they compute an explicit map from $\mathcal{C}$ to $J_\mathcal{D}$ by solving some system of differential equations. The total complexity of their algorithm is $\tilde{O}(\ell^2)$.

Some improvements to Couveignes and Ezome's method has been made by Milio \cite{EneaMilio}. He notices that in order to compute an equation for $\mathcal{D}$, it is not necessary to compute the equation of the quartic defining the Kummer surface. One just needs the images of two-torsion points, so the number of evaluations can be reduced. However, he does not compute the images of two-torsion points directly because the algorithm to evaluate functions does not behave well. This difficulty is handled with the $(16,6)$-configuration of the  Kummer surface.
Continuing the same ideas, it is possible to compute an equation for $\mathcal{D}$ of genus $g\ge3$ when $\mathcal{D}$ is hyperelliptic. In particular, Milio computes the equation for the genus $3$ case by using the $(64,29)$-configuration of the Kummer threefold. For the genus $3$ case where $\mathcal{D}$ is non-hyperelliptic, he applies the classical results of Weber \cite{Weber} and Riemman \cite{Riemann}, but
no explicit map from $\mathcal{C}$ to $J_\mathcal{D}$ is given.

For $g=3$ and $\ell=2$, Smith \cite{smith222} finds a rapidly computable isogeny from the Jacobian of hyperelliptic curve $\mathcal{C}$ to the Jacobian of non-hyperelliptic curve $\mathcal{D}$.
His construction is based on the so-called "trigonal construction" of Recillas \cite{Recillas}, Donagi and Livn\'e \cite{DL}, and it was later shown by Frey and Kani \cite{Frey-Kani} that Smith's result is a consequence of the study of moduli spaces for covers of the projective line with given ramification type and monodromy group.

In this paper, we will continue in the direction of \cite{Couveignes-Ezome,EneaMilio}: that is, we will address the problem of computing $(\ell,\ell,\ell)$-isogenies between
Jacobian varieties of dimension $3$. For our approach, an algebraic version of ratios of the $N$th order theta functions is of importance.
This material is explained by Shepherd-Barron in \cite{Barron}, and is introduced in Section 2.

In Section 3, we recall how to describe Kummer varieties
with Schr\"odinger coorditates. The usefulness of the Schr\"odinger coordinates is that it allows us to compute the action of the subgroup of two-torsion points
just by permuting the coordinates up to sign. In particular, one can easily obtain the coordinates of the images in the Kummer variety of all two-torsion points from those of the identity element.

In Section 4, we use formulae from the theory of classical theta functions to compute an equation for curve $\mathcal{D}$ (as Milio does in \cite{EneaMilio}).
Then we compute the linear change of coordinates to identify the Kummer varies of $J_\mathcal{C}/V$ and $J_\mathcal{D}$. It remains to be able to lift a point in the Kummer variety of $J_\mathcal{D}$ to the Jacobian $J_\mathcal{D}$. For this we compute eight "good" functions on the Kummer variety of $J_\mathcal{D}$ which allow us to construct a system of polynomial equations of low degree.
This system is easy to solve in practice, and a solution provides the required lifting. Though a point is lifted to two opposite points,
we emphasize that for application to cryptography this is sufficient. In fact, we can go further and compute an explicit expression for the $(\ell,\ell,\ell)$-isogeny by applying the same idea of \cite{Couveignes-Ezome}.

In Section 5, we explain briefly how to evaluate the functions we considered. In Section 6, we give an explicit example of computation.

\section{Theta functions}
In this section we recall the expressions for algebraic theta functions in terms of determinants of rational functions on the curve, upon which we will base future calculations.
%This material is explained in \cite{Barron}. The last subsection is about the Schr\"odinger coordinates associated to a Kummer threefold, which is taken from \cite{kummer3}.

\subsection{Weil functions}\label{weilfunc}
First we fix some notation. Let $\mathcal{C}$ be a projective, nonsingular curve of genus $g$ over an algebraically closed field $k$. If $D$ is a divisor on $\mathcal{C}$,
then we denote by $[D]$ the linear equivalence class of $D$. For every integer $n$ we denote by $\text{Pic}^n(\mathcal{C})$ the degree $n$ part of the divisor class group $\text{Pic}(\mathcal{C})$ of $\mathcal{C}$. The theta divisor on $\text{Pic}^{g-1}(\mathcal{C})$ is denoted by $W$. It is the set of all classes of effective divisors of degree $g-1$.
The canonical class on $\mathcal{C}$ is denoted by $K_\mathcal{C}$. If $Z$ is a divisor on $\text{Pic}(\mathcal{C})$ and $u$ is a point on $\text{Pic}(\mathcal{C})$, then we denote by $Z_u$ the translation of $Z$ by $u$.

%Let $\mathcal{C}$ be a projective, nonsingular curve of genus $g$ over an algebraically closed field $k$. Let $D$ be a divisor of degree $2g-1$ on $\mathcal{C}$ and let $(f_i^D)_{1\leq i\leq g}$ be a basis of the vector space $H^0(\mathcal{C},\mathcal{O}_\mathcal{C}(D))$. Then we have a nonzero section
%$\det(f_i^D(z_j))$ in $H^0(\mathcal{C}^g,\mathcal{O}_{\mathcal{C}^g}(\sum_{i=1}^g pr_i^*D))$,
%where the $pr_i$ are the projections of $\mathcal{C}^g$ onto the $g$ factors. We will use these determinants to form rational functions of the symmetric product $\mathcal{C}^{(g)}$, which can be thought of as functions on the Jacobian $J_\mathcal{C}=\text{Pic}^0(\mathcal{C})$.

Let $D$ be a divisor of degree $2g-1$ on $\mathcal{C}$ and let $(f_i^D)_{1\leq i\leq g}$ be a basis of the vector space $H^0(\mathcal{C},\mathcal{O}_\mathcal{C}(D))$. Then we have a nonzero section
$\det(f_i^D(z_j))$ in $H^0(\mathcal{C}^g,\mathcal{O}_{\mathcal{C}^g}(\sum_{i=1}^g pr_i^*D))$,
where the $pr_i$ are the projections of $\mathcal{C}^g$ onto the $g$ factors. We will use these determinants to form rational functions on the symmetric product $\mathcal{C}^{(g)}$, which can be thought of as functions on the Jacobian $J_\mathcal{C}=\text{Pic}^0(\mathcal{C})$.

Let $j_{g}:\mathcal{C}^{g}\to \text{Pic}^{g}(\mathcal{C})$
be the map which takes $(z_1,\ldots,z_{g})$ to the class of divisor $(z_1)+\cdots+(z_{g})$.
Let $\Delta_{ij}$ be the $(i,j)$ diagonal in $\mathcal{C}^g$. Then the rational function $\det(f_i^D(z_j))$ on $\mathcal{C}^g$
has divisor
\begin{equation}\label{zerolocus}
\text{div}(\det(f_i^D(z_j)))=j_{g}^*(W_{[D]-K_\mathcal{C}})+\sum_{1\leq i<j\leq g}\Delta_{ij}-\sum_{1\leq i\leq g}pr_i^*(D).
\end{equation}
This equation is used in \cite{Couveignes-Ezome,EneaMilio} and is proved by Shepherd-Barron in \cite[Proposition 4.1]{Barron}.
Roughly speaking, one can first check that set-theoretically, the zero locus of
$\det(f_i^D(z_j))$ equals $j_{g}^{-1}(W_{[D]-K_\mathcal{C}})\cup\cup_{1\leq i<j\leq g}\Delta_{ij}$, so
\[\text{div}(\det(f_i^D(z_j)))+\sum_{1\leq i\leq g}pr_i^*(D)=aj_{g}^*(W_{[D]-K_\mathcal{C}})+b\sum_{1\leq i<j\leq g}\Delta_{ij}
\]
for some $a,b>0$. Then one can consider the intersections of this effective divisor with curves $\{(z,\ldots,z)\in \mathcal{C}^g:z\in \mathcal{C}\}$ and
$\{(z_1,\ldots,z_{g-1},z):z\in \mathcal{C}\}$ (the $z_1,\ldots, z_{g-1}$ are fixed and distinct) respectively to determine that the coefficients $a,b$ are both equal to $1$.

Let $N\ge2$ be an integer that is prime to the characteristic of $k$ and $J_\mathcal{C}[N]$ the $N$-torsion subgroup of the Jacobian $J_\mathcal{C}$.
Fix a closed point $O\in \mathcal{C}$ and a divisor $\theta$ such that $2[\theta]=K_\mathcal{C}$.
For each $P\in J_\mathcal{C}[N]$, choose an effective divisor $E_P$ such that $[E_P-(O)-\theta]$ represents $P$. Then for each $P$, there is a rational function $h^{E_P}$ on $\mathcal{C}$ whose divisor is $N(E_P-(O)-\theta)$.
 Assume that $E_P=(O)+\theta$ for $P=0$.
Set $D_P=E_P+\theta$ and choose a basis $(f_i^{D_P})_{1\leq i\leq g}$ of $H^0(\mathcal{C}, \mathcal{O}_{\mathcal{C}}(D_P))$.
Then by \eqref{zerolocus}, the function
\begin{equation}\label{WeilFunction}
  g_P=(\frac{\det(f_i^{D_P}(z_j))}{\det(f_i^{D_0}(z_j))})^N\prod_{i=1}^{g}h^{E_P}(z_i)
\end{equation}
has divisor
\[\text{div}(g_P)=j_{g}^*(N W_{[E_P-\theta]}-N W_{[(O)]}).\]
Since permutations of factors of $\mathcal{C}^{g}$ leave $g_P$ invariant, $g_P$ is the pullback of a rational function on
$\text{Pic}^{g}(\mathcal{C})$. We set $\mathcal{W}=W_{-[\theta]}$. Then via the map $j_g$ followed by the translation $\text{Pic}^g(\mathcal{C})\to\text{Pic}^0(\mathcal{C})$ by $[-(O)-\theta]$, we see that $g_P$ is the pullback of a rational function $\eta_{P}$ on $\text{Pic}^{0}(\mathcal{C})$ with divisor $\text{div}(\eta_{P})=N \mathcal{W}_P-N \mathcal{W}$.
We say that $\eta_{-P}$ is a Weil function for $P$ (at level $N$).

%Let $N\ge2$ be an integer that is prime to the characteristic of $k$ and $J_\mathcal{C}[N]$ the $N$-torsion group of the Jacobian $J_\mathcal{C}$. Fix a point $O\in \mathcal{C}$. For each $P\in J_\mathcal{C}[N]$, choose an effective divisor $E_P$ such that $[E_P-g(O)]$ represents $P$. Then for each $P$, there is a rational function $h^{E_P}$ on $\mathcal{C}$ whose divisor is $N(E_P-g(O))$.
% Assume that $E_P=g(O)$ for $P=0$.
%Set $D_P=E_P+K_\mathcal{C}-(g-1)(O)$ and choose a basis $(f_i^{D_P})_{1\leq i\leq g}$ of $H^0(\mathcal{C}, \mathcal{O}_{\mathcal{C}}(D_P))$.
%Then by (\ref{zerolocus}), the function
%\[g_P=(\frac{\det(f_i^{D_P}(z_j))}{\det(f_i^{D_0}(z_j))})^N\prod_{i=1}^{g}h^{E_P}(z_i)\]
%has divisor
%\[\text{div}(g_P)=j_{g}^*(N W_{[E_P-(g-1)(O)]}-N W_{[(O)]}).\]
%Since permutations of factors of $\mathcal{C}^{g}$ leave $g_P$ invariant, $g_P$ is the pullback of a rational function on
%$\text{Pic}^{g}(\mathcal{C})$. We set $\mathcal{W}=W_{-[(g-1)(O)]}$. Then via the map $j_g$ followed by the translation $\text{Pic}^g(\mathcal{C})\to\text{Pic}^0(\mathcal{C})$ by $[-g(O)]$, we see that $g_P$ is the pullback of a rational function $\eta_{P}$ on $\text{Pic}^{0}(\mathcal{C})$ whose divisor is $\text{div}(\eta_{P})=N \mathcal{W}_P-N \mathcal{W}$.
%We say that $\eta_{-P}$ is a Weil function for $P$ (at level $N$).

\subsection{Algebraic theta functions}\label{sec:normalization}
%Let $G_N=k^*\times (\Z/N\Z)^g\times \mu_N^g$
%In subsection \ref{weilfunc} we constructed an explicit Weil function $f_P$, up to a scalar multiple, for every $N$-torsion point $P$ on $J_\mathcal{C}$. Here we compute normalized Weil functions

In subsection \ref{weilfunc} we constructed for every $N$-torsion point an explicit Weil function (up to a non-zero scalar multiple).
Here we want to choose normalized Weil functions which will be defined below.

%according to a symmetric theta structure
%Let $H(N)=(\Z/N\Z)^g\times \text{Hom}((\Z/N\Z)^g,k^*)$
%Let $\mathscr{G}_N$ be the set $k^*\times (\Z/N\Z)^g\times \mu_N^g$.

We recall Mumford's theory of theta group \cite{EqnAVI,AV}. Let $\mathcal{L}=\mathcal{O}_{J_\mathcal{C}}(N \mathcal{W})$ be the sheaf defined by the divisor $N \mathcal{W}$ on $J_\mathcal{C}$.
The theta group $\mathscr{G}(\mathcal{L})$ of $(J_\mathcal{C}, \mathcal{L})$ can be regarded as the set
\[\mathscr{G}(\mathcal{L})=\{(P,f_P): \text{div}(f_P)=NT_P^*\mathcal{W}-N\mathcal{W}, P\in J_\mathcal{C}[N] \}\]
with group law
\[(P,f_P)(Q,f_Q)=(P+Q, f_PT_P^*f_Q),\]
where $T_P:J_\mathcal{C}\to J_\mathcal{C}$ is the translation by $P$.
The kernel of the natural projection $\mathscr{G}(\mathcal{L})\to J_\mathcal{C}[N]$ can be  identified with $k^*$. So $\mathscr{G}(\mathcal{L})$ is a central extension
\[1\to k^*\to \mathscr{G}(\mathcal{L})\to J_\mathcal{C}[N]\to 0. \]
Note that if we choose a Weil function $f_P$ for each $P\in J_\mathcal{C}[N]$, then the associativity of the group law implies that
\[d:J_\mathcal{C}[N]\times J_\mathcal{C}[N]\to k^*, (P,Q)\mapsto\frac{f_PT_P^*f_Q}{f_{P+Q}},\]
is a $2$-cocycle (for trivial group action of $J_\mathcal{C}[N]$ on $k^*$). The Weil pairing $e_N$ can be expressed as
\begin{equation}\label{WeilPairing}
  e_N(P,Q)=\frac{d(P,Q)}{d(Q,P)}.
\end{equation}

%\[e_N(P,Q)=\frac{d(P,Q)}{d(Q,P)}=(P,f_P)(Q,f_Q)(P,f_P)^{-1}(Q,f_Q)^{-1}.\]

%Let $H_N=(\Z/N\Z)^g\times \text{Hom}((\Z/N\Z)^g,k^*)$ and define a group law on $k^*\times H_N$ by setting
%\[(\lambda_1,x_1,l_1)\cdot (\lambda_2,x_2,l_2)=(\lambda_1\lambda_2l_2(x_1),x_1+x_2,l_1+l_2).\]
%This law makes $k^*\times H(N)$ a group $\mathscr{G}_N$ which is an extension of $H_N$ by $k^*$. A theta structure on $(J_\mathcal{C}, \mathcal{L})$ is an isomorphism
%$\phi: \mathscr{G}_N\to \mathscr{G}(\mathcal{L})$ which is the identity on the center $k^*$. One immediately checks that $\phi$ induces an isomorphism $\bar{\phi}:H_N\to J_\mathcal{C}[N]$
%which takes the standard symplectic pairing to the Weil pairing.

%It is well-known that the commutator pairing determines the Weil pairing
%\[e_N(P,Q)=(P,f_P)(Q,f_Q)(P,f_P)^{-1}(Q,f_Q)^{-1}=\frac{f_PT_P^*f_Q}{f_QT_Q^*f_P}.\]
%Note that a choice of the $f_P$ determines a $2$-cocycle $d:J_\mathcal{C}[N]\times J_\mathcal{C}[N]\to k^*$ given by
%\[d(P,Q)=\frac{f_PT_P^*f_Q}{f_{P+Q}}\]
%\[(P,f_P)(Q,f_Q)=(P+Q, d(P,Q)f_{P+Q}).\]

We are just interested in a choice such that the induced $2$-cocycle $d$ is equal to a given bilinear pairing $d_N$ on $J_\mathcal{C}[N]$.
As in \cite{Barron} we call an ordered set $\{f_P\}_{P\in J_\mathcal{C}[N]}$ of Weil functions a \emph{Weil set}, and say it is \emph{normal} if $f_PT_P^*f_Q=d_N(P,Q)f_{P+Q}$ for all $P,Q\in J_\mathcal{C}[N]$.
An ordered set of non-zero scalars $\{\alpha_P\}_{ P\in J_\mathcal{C}[N]}$ is then called a \emph{normalization} of a given Weil set $\{f_P\}_{P\in J_\mathcal{C}[N]}$ if $\{\tilde{f}_P=\alpha_Pf_P\}_{P\in J_\mathcal{C}[N]}$ is normal.
%The reason for these notions is that there is a close connection between normal Weil sets and theta structures.
%We can endow the set $k^*\times J_\mathcal{C}[N]$ with a group law by setting $(\lambda,P)(\mu,Q)=(\lambda\mu d_N(P,Q),P+Q)$. So it is an extension analogous to the theta group $\mathscr{G}(\mathcal{L})$.
%A \emph{theta structure} on $(J_\mathcal{C}, \mathcal{L})$ is then an isomorphism (of group) $\phi: k^*\times J_\mathcal{C}[N]\to \mathscr{G}(\mathcal{L})$ which is the identity on $k^*$.
%Notice that a theta structure $\phi$ induces an isomorphism $\bar{\phi}:J_\mathcal{C}[N]\to J_\mathcal{C}[N]$, so we have in fact isomorphic extensions:
%\[\xymatrix{1\ar[r] &k^*\ar[r]\ar@{=}[d]  &k^*\times J_\mathcal{C}[N] \ar[r]\ar[d]^\phi&  J_\mathcal{C}[N]\ar[r]\ar[d]^{\bar{\phi}} &0\\
%1\ar[r] &k^*\ar[r]  &\mathscr{G}(\mathcal{L}) \ar[r]&  J_\mathcal{C}[N]\ar[r] &0.
%}\]
%For each $P\in J_\mathcal{C}[N]$, we have a unique Weil function $f_{\bar{\phi}(P)}$ for $\bar{\phi}(P)$ satisfying $\phi((1,P))=(\bar{\phi}(P),f_{\bar{\phi}(P)})$.
%One can check that the corresponding Weil set $\{f_P\}_{P\in J_\mathcal{C}[N]}$ is normal with respect to the bilinear pairing $d_N\circ (\bar{\phi}^{-1}\times \bar{\phi}^{-1})$.
%This implies that if we know $\bar{\phi}$, then we can recover the theta structure $\phi$ from a Weil set up to an automorphism of $\mathscr{G}(\mathcal{L})$ which induces the identity on both $k^*$ and $J_\mathcal{C}[N]$. All possible automorphisms arise from some character $\chi\in Hom(J_\mathcal{C}[N],K^*)$ (see \cite[Lemma 3.2]{Barron}).
Two kinds of $d_N$ which will be used later are the following \cite{Barron}. If $P_1,\ldots,P_{2g}$ is a symplectic basis of $J_\mathcal{C}[N]$, then
$d_N(P_i,P_{i+g})=e_N(P_i,P_{i+g})$ for $i=1,\ldots,g$ and $d_N=1$ on all other pairs of basis elements. In case $N$ is odd, we will often take $d_N(P,Q)=e_N(P,Q)^{(N+1)/2}$.

Now we recall from \cite[Section 3.2]{Barron} the method of computing a normalization of a Weil set.
Define $\gamma(P,Q)=d_N(P,Q)\frac{f_{P+Q}}{f_PT_P^*f_Q}$ for all $P,Q\in J_\mathcal{C}[N]$. Then we have to solve the equations  \[\frac{\alpha_P\alpha_Q}{\alpha_{P+Q}}=\gamma(P,Q).\]
Since $\tilde{f}_P=1$, we can assume $\alpha_0=f_0=1$, so
\begin{equation}\label{Npower}
  \alpha_P^N=\gamma(P,P)\gamma(P,2P)\cdots\gamma(P,(N-1)P).
\end{equation}
Let $P_1,\ldots,P_{2g}$ be a symplectic basis of $J_\mathcal{C}[N]$. For each $i=1,\ldots,2g$,
take a root $\alpha_{P_i}$ of equation \eqref{Npower} arbitrarily. Then for any integer $j$,
\[\alpha_{(j+1)P_i}=\frac{\alpha_{jP_i}\alpha_{P_i}}{\gamma(jP_i,P_i)},\]
so we obtain $\alpha_P$ for all $P\in \langle P_i\rangle$.
Next, suppose we know $\alpha_P$ for all $P\in \langle P_1,\ldots,P_i\rangle$, then
\[\alpha_{P+jP_{i+1}}=\frac{\alpha_{jP_{i+1}}\alpha_P}{\gamma(jP_{i+1},\alpha_P)},\]
so by induction on $i$ we obtain all $\alpha_P$.

Note that $\mathcal{W}$ is symmetric. If we require that $\tilde{f}_P$ be symmetric, namely,
$[-1]^*\tilde{f}_P=\tilde{f}_{-P}$, then
\[\alpha_P^2=\gamma(P,-P)\frac{f_{-P}}{[-1]^*f_P}.\]
So if $N$ is odd, then $\alpha_P=\alpha_P^N /(\alpha_P^2)^{\frac{N-1}{2}}$ is determined uniquely, and if $N$ is even,
there is an ambiguity of $\pm1$.

\begin{remark}
By definition, the map $\gamma:J_\mathcal{C}[N]\times J_\mathcal{C}[N]\to k^*$ is a $2$-coboundary, which means that there exists a map $\alpha:J_\mathcal{C}[N]\to k^*$ such that
$\gamma(P,Q)=\alpha(P)\alpha(Q)\\ \alpha(P+Q)^{-1}$. The fact that $\gamma$ is $2$-cocycle also implies that the values of $\alpha(P+Q+R)$ computed respectively from $\alpha(P),\alpha(Q+R)$ and $\alpha(P+Q),\alpha(R)$ coincide. So $\alpha$, as we have seen, is determined by its values on a basis of $J_\mathcal{C}[N]$.
\end{remark}

\begin{remark}
 If $\{\tilde{f}_P\}_{P\in J_\mathcal{C}[N]}$ is a normal Weil set, then for each character $\chi\in Hom(J_\mathcal{C}[N],k^*)$, $\{\chi(P)\tilde{f}_P\}_{P\in J_\mathcal{C}[N]}$ is also a normal Weil set. In fact,
 these are all possible normal Weil sets (with respect to $d_N$)(see \cite[Lemma 3.2]{Barron}).
 If furthermore each $\tilde{f}_P$ is symmetric, then it is not necessary that all $\chi(P)\tilde{f}_P$ are symmetric. For odd $N$, it is the case only for trivial $\chi$, and for $N=2$ for all characters.
\end{remark}

%Let $\{\tilde{f}_P\}_{P\in J_\mathcal{C}[N]}$ and $\{\tilde{f}_P'\}_{P\in J_\mathcal{C}[N]}$ be normal Weil sets with respect to bilinear pairings $d_N$ and $d_N'$ respectively.
%Then we can compute a normalization $\{\alpha_P''\}_{P\in J_\mathcal{C}[N]}$ of $\{\tilde{f}_P\}_{P\in J_\mathcal{C}[N]}$ with respect to $d_N'$.

\subsection{Classical theta functions }
In this subsection we recall normalized Weil functions' link with analytic theta functions with characteristics when the base field is $\C$.

%We will use some formulae from the theory of analytic theta functions for computation.

Now suppose that $k=\C$. Once a symplectic basis $\gamma_1,\ldots,\gamma_g,\gamma_1',\ldots,\gamma_g'$ of $H_1(\mathcal{C},\Z)$ is chosen, we can find a normalized basis $\omega_1,\ldots,\omega_g$ of $H^0(\mathcal{C},\Omega^1)$ such that the $g\times 2g$ period matrix $(\int_{\gamma_j}\omega_i, \int_{\gamma_j'}\omega_i)$ is equal to $(I_g,\Omega)$ with $\Omega\in \mathfrak{H}_g$, the
Siegel upper-half-space of degree $g$. The column vectors of this matrix generate a lattice $\Lambda_\Omega=\Z^g+\Omega\Z^g$ in $\C^g$, and the complex torus $\C^g/\Lambda_\Omega$ is called the \emph{analytic Jacobian variety} of $\mathcal{C}$.
%The complex torus $\C^g/\Lambda_\Omega$ with $\Lambda_\Omega=\Z^g+\Omega\Z^g$ is called the Jacobian variety of $\mathcal{C}$.
If $P_0$ is a base point on $\mathcal{C}$, then we have the Abel-Jacobi map
\[A:\mathcal{C}\to \C^g/\Lambda_\Omega, P\mapsto (\int_{P_0}^P\omega_1,\ldots, \int_{P_0}^P\omega_g) \mod \Lambda_\Omega.\]
This map extends $\Z$-linearly to arbitrary divisors, and then induces an isomorphism (of abelian groups) $A: \text{Pic}^0(\mathcal{C})\to \C^g/\Lambda_\Omega$.
The Weil pairing on $J_\mathcal{C}[N]$ is then identified with
\[\langle ,\rangle: \frac{1}{N}\Lambda_\Omega/\Lambda_\Omega\times \frac{1}{N}\Lambda_\Omega/\Lambda_\Omega\to \mu_N \]
given by $\langle \Omega x_1+y_1+\Lambda_\Omega, \Omega x_2+y_2+\Lambda_\Omega\rangle= \exp(-2\pi i N({}^t\!x_1y_2-{}^t\!x_2y_1))$.

Recall that a theta characteristic of $\mathcal{C}$ is a divisor class $[D]$ such that $2[D]$ is the canonical divisor class.
%Each theta characteristic $[D]$ determines a quadratic form on $J_\mathcal{C}[2]$ by setting
% \[q_{[D]}([D'])=\text{dim}\,H^0(\mathcal{C},\mathcal{O}_\mathcal{C}(D))+\text{dim}\,H^0(\mathcal{C},\mathcal{O}_\mathcal{C}(D+D')) \mod 2.\]
To any theta characteristic $[D]$ one can attach the quadratic form $q_{[D]}$ on $J_\mathcal{C}[2]$ by setting
 \[q_{[D]}([D'])=\text{dim}\,H^0(\mathcal{C},\mathcal{O}_\mathcal{C}(D))+\text{dim}\,H^0(\mathcal{C},\mathcal{O}_\mathcal{C}(D+D')) \mod 2.\]
A theta characteristic is called even (resp. odd) if the Arf invariant of the associated quadratic form is equal to $0$ (resp. $1$).

The chosen symplectic basis of $H^0(\mathcal{C},\Omega^1)$ determines a level $2$ structure $\phi:J_\mathcal{C}[2]\to (\frac{1}{2}\Z/\Z)^{2g}$, which then determines a unique even theta characteristic $\delta$ such that if $[D]$ is a theta characteristic with $\phi([D]-\delta)=(m_1,m_2)$, then $[D]$ is even if and only if the integer $4 \,{}^t\!m_1m_2$ is even \cite[Lemma 3.19]{Barron}.
%$e([D])=(4 \,{}^t\!m_1m_2)\mod 2$.
Riemann's theorem says that $A(W_{-\delta})$ is the theta divisor $\Theta$ induced by the zero locus of theta function
\[\theta\begin{bsmallmatrix}0\\ 0\end{bsmallmatrix}(z,\Omega)=\sum_{n\in \Z^g}\exp\,(\pi i {}^t\!n\Omega n+2\pi i {}^t\!nz).\]
Note that for all $m,n \in \Z^g$,
\[\theta\begin{bsmallmatrix}0\\ 0\end{bsmallmatrix}(z+\Omega m+n,\Omega)=\exp(-\pi i{}^t\!m\Omega m-2\pi i{}^t\!mz)\theta\begin{bsmallmatrix}0\\ 0\end{bsmallmatrix}(z,\Omega).\]
This implies that the zero locus of $\theta\begin{bsmallmatrix}0\\ 0\end{bsmallmatrix}(z,\Omega)$ is invariant under translations by elements of
$\Lambda_\Omega$ and therefore induces a divisor on $\C^g/\Lambda_\Omega$.

For column vectors $a,b\in \Q^g$, the theta
function with characteristic $\begin{bsmallmatrix}a\\b\end{bsmallmatrix}$ is defined to be
\[\theta\begin{bsmallmatrix}a\\ b\end{bsmallmatrix}(z,\Omega)=\exp(\pi i{}^t\!a\Omega a+2\pi i{}^t\!a(z+b))\theta\begin{bsmallmatrix}0\\ 0\end{bsmallmatrix}(z+\Omega a+b,\Omega).\]
%\[\theta\begin{bsmallmatrix}a\\ b\end{bsmallmatrix}(z,\Omega)=\sum_{n\in \Z^g}\exp\,(\pi i {}^t\!(n+a)\Omega (n+a)+2\pi i {}^t\!(n+a)(z+b)).\]
It is easy to check that for all $m,n \in \Z^g$,
\[\theta\begin{bsmallmatrix}a\\ b\end{bsmallmatrix}(z+\Omega m+n,\Omega)=\exp(-\pi i{}^t\!m\Omega m-2\pi i{}^t\!mz)\exp(2\pi i ({}^t\!an-{}^t\!b m))\theta\begin{bsmallmatrix}a\\ b\end{bsmallmatrix}(z,\Omega)\]
and
\[\theta\begin{bsmallmatrix}a+m\\ b+n\end{bsmallmatrix}(z,\Omega)=\exp(2\pi i {}^t\!an)\theta\begin{bsmallmatrix}a\\ b\end{bsmallmatrix}(z,\Omega).\]
So the function $\theta\begin{bsmallmatrix}a\\ b\end{bsmallmatrix}(z,\Omega)$ also defines a divisor on $\C^g/\Lambda_\Omega$.

Let $a,b\in (\frac{1}{N}\Z)^g$.
 Then $\Omega a+b$ represents an $N$-torsion point of $\C^g/\Lambda_\Omega$, and
 \begin{equation}
   f_{(a,b)}=(\frac{\theta\begin{bsmallmatrix}a\\ b\end{bsmallmatrix}(z,\Omega)}{\theta\begin{bsmallmatrix}0\\ 0\end{bsmallmatrix}(z,\Omega)})^N
 \end{equation}
is a well-defined function on $\C^g/\Lambda_\Omega$ with divisor $\text{div}(f_{(a,b)})=N T_{\Omega a+b}^*\Theta- N\Theta$.
Note that $f_{(a,b)}$ depends only on the classes of $a,b\in (\frac{1}{N}\Z)^g/\Z^g$. If we take the bilinear pairing $d_N$ on $J_\mathcal{C}[N]=\frac{1}{N}\Lambda_\Omega/\Lambda_\Omega$
defined by $d_N(\Omega x_1+y_1,\Omega x_2+y_2)=\exp(-2\pi i N {}^t\!x_1y_2)$, then $\langle P,Q\rangle=d_N(P,Q)/d_N(Q,P)$, and $\{f_{(a,b)}\}_{a,b\in (\frac{1}{N}\Z)^g/\Z^g}$ is easily checked to be
a normal Weil set with respect to $d_N$.

%Let $\mathfrak{H}_g$ denote the Siegel upper-half-space of degree $g$, namely, the set of symmetric $g\times g$ complex matrices with positive definite imaginary part.
%
%Then we have a commutative diagram
%\[\xymatrix{1\ar[r]&k^*\ar[r]&G(\mathcal{O}(NX))\ar[r]& \text{Pic}^{0}(C)[N]\ar[r]&0\\
%1\ar[r]&k^*\ar[r]\ar[u]^{id}&k^*\times\text{Pic}^{0}(C)[N]\ar[r]\ar[u]^{(\lambda,P)\mapsto(P,\lambda\tilde{f}_P)}_{\cong}& \text{Pic}^{0}(C)[N]\ar[r]\ar[u]_{P\mapsto P}&0
%}\]

\section{Kummer variety}\label{kummer_variety_intro}
In this section we will mention a few basic results about Kummer varieties.

Let $P_1,\ldots,P_{2g}$ be a symplectic basis of $J_\mathcal{C}[2]$. Let $d$ be the bilinear pairing on $J_\mathcal{C}[2]$ which takes $-1$ on pairs $(P_i,P_j)$ of basis elements exactly for $j=g+i,i=1,\ldots, g$. Let $\{\tilde{f}_P\}_{P\in J_\mathcal{C}[2]}$ be a normal Weil set with respect to the bilinear pairing $d$. Assume that $\sum_{P\in \langle P_{g+1},\ldots,P_{2g}\rangle} \tilde{f}_P\ne 0$. Then
from the proof of \cite[Proposition 3, page 296]{EqnAVI}, we have the following proposition.
\begin{proposition}
For $P\in \langle P_1,\ldots,P_g\rangle$, let $\theta_P=\sum_{Q\in \langle P_{g+1},\ldots,P_{2g}\rangle} d(P,Q)\tilde{f}_{P+Q}$. Then
$\{\theta_P:P\in \langle P_1,\ldots,P_g\rangle\}$ is a basis of $H^0(J_\mathcal{C}, \mathcal{O}_{J_\mathcal{C}}(2 \mathcal{W}))$, and the
action of the theta group on $H^0(J_\mathcal{C}, \mathcal{O}_{J_\mathcal{C}}(2 \mathcal{W}))$ is given by
\[(Q,\tilde{f}_Q).\theta_P=d(Q'+P,Q'')\theta_{P+Q'},\]
where $Q=Q'+Q''$ with $Q'\in \langle P_1,\ldots,P_g\rangle$ and $Q''\in \langle P_{g+1},\ldots,P_{2g}\rangle$.
\end{proposition}

Let $\mathcal{L}=\mathcal{O}_{J_\mathcal{C}}(2 \mathcal{W})$.
For a closed point $a\in J_\mathcal{C}$, let $\mathcal{L}(a)=\mathcal{L}_a/\mathfrak{m}_a\mathcal{L}_a$, where $\mathcal{L}_a$ is the stalk of $\mathcal{L}$ at $a$ and $\mathfrak{m}_a$
is the maximal ideal of the local ring at $a$. Then we have a canonical map
\[\kappa: J_\mathcal{C}\to \mathbb{P}^{2^g-1},a\mapsto [\epsilon_a(\theta_P|_a)]_{P\in \langle P_1,\ldots,P_g\rangle },\]
where $\epsilon_a$ is an isomorphism $\mathcal{L}(a)\to k$ and $\theta_P|_a$ denotes the image of $\theta_P$ in $\mathcal{L}(a)$.

Let $a\in J_\mathcal{C}$ and $w=w_1+w_2\in J_\mathcal{C}[2]$ with $w_1\in \langle P_1,\ldots,P_g\rangle$ and $w_2\in \langle P_{g+1},\ldots,P_{2g}\rangle$. There is a convenient way to compute $\kappa(a+w)$ from $\kappa(a)$ just by permuting the coordinates up to sign. Choose an isomorphism $\epsilon_a:\mathcal{L}(a)\to k$. Then
we can define the composite isomorphism
\[\epsilon_{a+w}:\mathcal{L}(a+w)=(T_w^*\mathcal{L})(a)\xrightarrow{(w,\tilde{f}_w)} \mathcal{L}(a)\xrightarrow{\epsilon_a} k,\]
It follows immediately that
\[\kappa(a+w)=[\epsilon_{a+w}(\theta_P|_{a+w})]_{P\in \langle P_1,\ldots,P_g\rangle }=[d(P,w_2)\epsilon_a(\theta_{P+w_1}|_a)]_{P\in \langle P_1,\ldots,P_g\rangle }.\]

For $k=\C$, we can express the above result in the language of analytic theta functions.

For $\Omega\in \mathfrak{H}_g$, the $2^g$ functions
\[X_{i_1\cdots i_g}(z)=\theta\begin{bsmallmatrix}{}^t(\frac{i_1}{2}&\frac{i_2}{2}&\cdots &\frac{i_g}{2})\\ {}^t(0&0&\cdots&0)\end{bsmallmatrix}(2z,2\Omega)\,\,(i_1,i_2,\ldots,i_g\in \{0,1\})\]
can be identified with a basis of $H^0(\C^g/\Lambda_\Omega,\mathcal{O}(2\Theta))$ \cite[Proposition 1.3, page 124]{Mumford-Theta1}. These are known as the Schr\"odinger coordinates, and
we can use them to define the Kummer map
\[\kappa_\Omega: \C^g/\Lambda_\Omega\to \mathbb{P}^{2^g-1},z+\Lambda_\Omega\mapsto [X_{00\cdots0}(z):X_{00\cdots01}(z):\cdots:X_{11\cdots1}(z)],\]
whose image is isomorphic to the Kummer variety of $\C^g/\Lambda_\Omega$. The action of the subgroup of two-torsion points can now be easily deduced from the "quasi-periodicity" property of theta functions with characteristics.

The normalised Weil function $\tilde{f}_P$ can be thought of as
\begin{equation*}
  (\frac{\theta\begin{bsmallmatrix}{}^t(\frac{i_1}{2}&\frac{i_2}{2}&\cdots&\frac{i_g}{2})\\ {}^t(\frac{i_{g+1}}{2}&\frac{i_{g+2}}{2}&\cdots&\frac{i_{2g}}{2})\end{bsmallmatrix}(z,\Omega)}{\theta\begin{bsmallmatrix}{}^t(0&0&\cdots&0)\\ {}^t(0&0&\cdots&0)\end{bsmallmatrix}(z,\Omega)})^2,
\end{equation*}
where $i_1,i_2,\ldots,i_{2g}\in\{0,1\}$ satisfy $P=i_1 P_1+\cdots+i_{2g} P_{2g}$.
To compute the corresponding projective point
\[[\theta\begin{bsmallmatrix}{}^t(0&0&\cdots&0)\\ {}^t(0&0&\cdots&0)\end{bsmallmatrix}(2z,2\Omega):\theta\begin{bsmallmatrix}{}^t(0&0&\cdots&\frac{1}{2})\\ {}^t(0&0&\cdots&0)\end{bsmallmatrix}(2z,2\Omega):
\cdots:\theta\begin{bsmallmatrix}{}^t(\frac{1}{2}&\frac{1}{2}&\cdots&\frac{1}{2})\\ {}^t(0&0&\cdots&0)\end{bsmallmatrix}(2z,2\Omega)],\]
we must convert the projective point
\[[\tilde{f}_{i_1 P_1+\cdots+i_{2g} P_{2g}}]_{(i_1,\ldots,i_{2g})\in\{0,1\}^{2g}}=[\theta\begin{bsmallmatrix}{}^t(\frac{i_1}{2}&\frac{i_2}{2}&\cdots&\frac{i_g}{2})\\ {}^t(\frac{i_{g+1}}{2}&\frac{i_{g+2}}{2}&\cdots&\frac{i_{2g}}{2})\end{bsmallmatrix}(z,\Omega)^2]_{(i_1,\ldots,i_{2g})\in\{0,1\}^{2g}}.\]
This can be achieved by the following proposition.
\begin{proposition}\label{to_Schroedinger}
Let $Rep(\frac{1}{2}\Z^g/\Z^g)=\{0,\frac{1}{2}\}^g$ be the set of representatives of $\frac{1}{2}\Z^g/\Z^g$. Then for
$a,b\in Rep(\frac{1}{2}\Z^g/\Z^g)$,
 \[
    2^g \theta\begin{bsmallmatrix}
      a\\0
    \end{bsmallmatrix}(2z,2\Omega)\theta\begin{bsmallmatrix}
      b\\0
    \end{bsmallmatrix}(0,2\Omega)
    =\sum_{\beta\in Rep(\frac{1}{2}\Z^g/\Z^g) }\exp(-4\pi i\cdot {}^t\beta a) \theta\begin{bsmallmatrix}
      a+b\\ \beta
    \end{bsmallmatrix}(z,\Omega)\theta\begin{bsmallmatrix}
      a-b\\ \beta
    \end{bsmallmatrix}(z,\Omega).
  \]
\end{proposition}
\begin{proof}
Apply the Theorem 1.3 of \cite{Koizumi} with the matrices $T=\begin{bsmallmatrix}
      \frac{1}{2}&\frac{1}{2}\\\frac{1}{2}&-\frac{1}{2}
    \end{bsmallmatrix}$ and $\Gamma=\begin{bsmallmatrix}
      2&0\\0&2
    \end{bsmallmatrix}$.
\end{proof}

For the other direction, we can use the following proposition.
\begin{proposition}\label{schroedinger_to}
Let $Rep(\frac{1}{2}\Z^g/\Z^g)=\{0,\frac{1}{2}\}^g$ be the set of representatives of $\frac{1}{2}\Z^g/\Z^g$. Then for
$a,b\in Rep(\frac{1}{2}\Z^g/\Z^g)$,
 \[ \theta\begin{bsmallmatrix}
      a\\b
    \end{bsmallmatrix}(z,\Omega)^2=\sum_{\beta\in Rep(\frac{1}{2}\Z^g/\Z^g) }(-1)^{4{}^t\beta b} \theta\begin{bsmallmatrix}
      \beta\\0
    \end{bsmallmatrix}(2z,2\Omega)\theta\begin{bsmallmatrix}
      \beta+a\\ 0
    \end{bsmallmatrix}(0,2\Omega)\]
\end{proposition}
\begin{proof}
Apply the Theorem 1.3 of \cite{Koizumi} with the matrices $T=\begin{bsmallmatrix}
      \frac{1}{2}&\frac{1}{2}\\\frac{1}{2}&-\frac{1}{2}
    \end{bsmallmatrix}$ and $\Gamma=\begin{bsmallmatrix}
      1&0\\0&1
    \end{bsmallmatrix}$. Then use \cite[Proposition 1.3, page 124]{Mumford-Theta1}.
\end{proof}

\begin{remark}
If $g=3$ and $\C^3/\Lambda_\Omega$ is the Jacobian variety of a non-hyperelliptic curve, then the image of $\kappa_\Omega$ can be described as the
singular locus of the Coble quartic hypersurface in $\mathbb{P}^7$ \cite{CobleQuartic}. The defining polynomial of this quartic hypersurface
can be written as
\begin{eqnarray*}
F_\Omega&=& r(X_{000}^4+X_{001}^4+X_{010}^4+X_{011}^4+X_{100}^4+X_{101}^4+X_{110}^4+X_{111}^4) \\
        && +\, t_{001}(X_{000}^2X_{001}^2+X_{010}^2X_{011}^2+X_{100}^2X_{101}^2+X_{110}^2X_{111}^2)\\
        && +\, t_{010}(X_{000}^2X_{010}^2+X_{001}^2X_{011}^2+X_{100}^2X_{110}^2+X_{101}^2X_{111}^2)\\
        &&+\, t_{011}(X_{000}^2X_{011}^2+X_{001}^2X_{010}^2+X_{100}^2X_{111}^2+X_{101}^2X_{110}^2)\\
        &&+\, t_{100}(X_{000}^2X_{100}^2+X_{001}^2X_{101}^2+X_{010}^2X_{110}^2+X_{011}^2X_{111}^2)\\
        &&+\, t_{101}(X_{000}^2X_{101}^2+X_{001}^2X_{100}^2+X_{010}^2X_{111}^2+X_{011}^2X_{110}^2)\\
        &&+\, t_{110}(X_{000}^2X_{110}^2+X_{001}^2X_{111}^2+X_{010}^2X_{100}^2+X_{011}^2X_{101}^2)\\
        &&+\, t_{111}(X_{000}^2X_{111}^2+X_{001}^2X_{110}^2+X_{010}^2X_{101}^2+X_{011}^2X_{100}^2)\\
        &&+\, s_{001}(X_{000}X_{010}X_{100}X_{110}+X_{001}X_{011}X_{101}X_{111})\\
        &&+\, s_{010}(X_{000}X_{001}X_{100}X_{101}+X_{010}X_{011}X_{110}X_{111})\\
        &&+\, s_{011}(X_{000}X_{011}X_{100}X_{111}+X_{001}X_{010}X_{101}X_{110})\\
        &&+\, s_{100}(X_{000}X_{001}X_{010}X_{011}+X_{100}X_{101}X_{110}X_{111})\\
        &&+\, s_{101}(X_{000}X_{010}X_{101}X_{111}+X_{001}X_{011}X_{100}X_{110})\\
        &&+\, s_{110}(X_{000}X_{001}X_{110}X_{111}+X_{010}X_{011}X_{100}X_{101})\\
        &&+\, s_{111}(X_{000}X_{011}X_{101}X_{110}+X_{001}X_{010}X_{100}X_{111}),
\end{eqnarray*}
where by abuse of notation we denote also by $X_{000},\ldots,X_{111}$ the Schr\"odinger coordinates on $\mathbb{P}^7$, and
$r,t_{001},\ldots,t_{111},s_{001},\ldots,s_{111}$ are coefficients. This expression can be found in
\cite{kummer3,CobleQuarticequation}. So to determine equations $\{\frac{\partial F_\Omega}{\partial X_{000}}=0,\ldots,\frac{\partial F_\Omega}{\partial X_{111}}=0\}$ for the Kummer variety, we need only know two points in general.
However, we do not need equations of the Kummer varieties, even in determining the projective change of coordinates which is used to
identify the Kummer varieties of $J_\mathcal{C}/V$ and $J_\mathcal{D}$ (see subsection \ref{image_computation}).
\end{remark}

\section{Quotients of Jacobians}
Let $\mathcal{C}$ be a hyperelliptic curve of genus $g$. Assume that $\mathcal{C}$ is given by $y^2=f(x)$, where $f$ is a polynomial of degree $2g+1$.
Let $\infty$ be the point at infinity. Then the divisor $\mathcal{W}=W_{-[(g-1)(\infty)]}$ on $J_\mathcal{C}$ is symmetric.

Let $V\subseteq J_\mathcal{C}[\ell]$ be a maximal isotropic subgroup with respect to the Weil pairing $e_\ell$. Let $\psi: J_\mathcal{C}\to J_\mathcal{C}/V$ be the quotient map.
Let $d_\ell$ be the bilinear pairing on $J_\mathcal{C}[\ell]$ given by $d_\ell(P,Q)=e_\ell(P,Q)^{(\ell+1)/2}$.
%Choose the bilinear pairing $d_\ell$ on $J_\mathcal{C}[\ell]$ such that $d_\ell(P,Q)=e_\ell(P,Q)^{(\ell+1)/2}$.
Then we have a normal Weil set $\{\tilde{f}_P\}_{P\in J_\mathcal{C}[\ell]}$ with the property that for each $P\in J_\mathcal{C}[\ell]$, $[-1]^*\tilde{f}_P=\tilde{f}_{-P}$.
One checks easily that
\[\tilde{K}=\{(P,\tilde{f}_P): P\in V\}\] is a level subgroup of $\mathscr{G}(\mathcal{L})$ for sheaf $\mathcal{L}=\mathcal{O}_{J_\mathcal{C}}(\ell \mathcal{W})$, so it
determines an invertible sheaf $\mathcal{M}$ on $J_\mathcal{C}/V$ such that $\psi^*\mathcal{M}\simeq \mathcal{L}$. This sheaf defines a symmetric principal polarization, so there is a unique effective divisor $Y$ on $J_\mathcal{C}/V$ with $\mathcal{M}\simeq \mathcal{O}_{J_\mathcal{C}/V}(Y)$. We set $X=\psi^*Y$. Then $X-\ell \mathcal{W}$ is the divisor of a rational function $\tilde{\theta}$, which by the theory of descent is a section of $\mathcal{L}$ invariant under $\tilde{K}$. Recall that the group $\mathscr{G}(\mathcal{L})$ acts on $H^0(J_\mathcal{C},\mathcal{L})$ by
\[(P,f_P).s=f_PT_P^*s.\]
Since $\text{dim}\,H^0(J_\mathcal{C},\mathcal{L})^{\tilde{K}}=\text{dim}\,H^0(J_{\mathcal{C}}/V,\mathcal{M})=1$, we can take
\[\tilde{\theta}=\sum_{P \in V }(P,\tilde{f}_P).1=\sum_{P\in V} \tilde{f}_P\] if it is non-zero.

\begin{remark}
The isotropic subgroup $V$ considered is assumed to be given as part of the input. In fact, we can compute a basis of the $\ell$-torsion subgroup by using the method of Couveignes \cite{Couveignes_torsion}, and
 transform it into a symplectic one for the Weil pairing by using a modified version of the Gram-Schmidt process. Then we can enumerate all maximal isotropic subgroups and the rational ones among them. For our choice of $d_\ell$, $\tilde{K}$ is always a level subgroup as long as $V$ is isotropic.
\end{remark}

Now we set $\mathcal{A}=J_\mathcal{C}/V$ and suppose that $(\mathcal{A},Y)$ is the Jacobian variety of a curve $\mathcal{D}$.

\subsection{Equation of $\mathcal{D}$}\label{subsec: equation_of_D}
We first compute the normalized Weil functions for points in $\mathcal{A}[2]$. Evaluating them at a certain point, we obtain the Moduli point.
Then we consider the case $g=3$, and derive an equation for $\mathcal{D}$ by using formulae from the theory of analytic theta functions.

Choose a Weil set $\{\eta_S\}_{S\in J_\mathcal{C}[2]}$, then it follows from
\[\text{div}(\eta_S^\ell T_S^*\tilde{\theta}^2/\tilde{\theta}^2)=\ell (2T_S^*\mathcal{W}-2\mathcal{W})+T_S^*(2X-2\ell \mathcal{W})-(2X-2\ell \mathcal{W})=2 T_S^*X-2X
\]
that there are Weil functions $g_{\bar{S}}$ with  $\text{div}(g_{\bar{S}})=2 T_{\bar{S}}^*Y-2Y$ such that $\psi^*g_{\bar{S}}=\eta_S^\ell T_S^*\tilde{\theta}^2/\tilde{\theta}^2$.
Here we denote by $\bar{S}$ the image of $S$ in $\mathcal{A}$.
Let $S_1,\ldots, S_{2g}$ be a symplectic basis of $J_\mathcal{C}[2]$. Then $e_2(S_i,S_j)=e_2(\bar{S}_i,\bar{S}_j)$, which implies that $\bar{S}_1,\ldots, \bar{S}_{2g}$ is a symplectic basis of $\mathcal{A}[2]$. Let $d'$ be the bilinear pairing on $J_\mathcal{C}[2]$ which takes $-1$ on pairs $(S_i,S_j)$ of basis elements exactly for $j=g+i,i=1,\ldots, g$.
Then it also can be viewed as a bilinear pairing on $\mathcal{A}[2]$ by setting $d'(\bar{S}_i,\bar{S}_j)=d'(S_i,S_j)$.
Using $d'$, we can compute a normalization $\{\beta_{\bar{S}}\}$ of the Weil set $\{g_{\bar{S}}\}_{S\in J_\mathcal{C}[2]}$.
As we have seen, this is equivalent to solving the equations
\[\frac{\beta_{\bar{P}}\beta_{\bar{Q}}}{\beta_{\bar{P}+\bar{Q}}}=\gamma'(\bar{P},\bar{Q}),\]
where $\gamma'(\bar{P},\bar{Q})=d'(\bar{P},\bar{Q})\frac{g_{\bar{P}+\bar{Q}}}{g_{\bar{P}}T_{\bar{P}}^*g_{\bar{Q}}}$. Note that $\gamma'(\bar{P},\bar{Q})\in k^*$,
so
\begin{equation}\label{normalization}
  \gamma'(\bar{P},\bar{Q})=\psi^*\gamma'(\bar{P},\bar{Q})=d'(P,Q)(\frac{\eta_{P+Q}}{\eta_PT_P^*\eta_Q})^\ell.
\end{equation}

The symplectic basis $S_1,\ldots, S_{2g}$ determines a unique element $S_\delta\in J_\mathcal{C}[2]$ such that if $\{\tilde{\eta}_S\}_{S\in J_\mathcal{C}[2]}$ is a normal Weil set
with respect to $d'$ and $S=a_1 S_1+\cdots+a_{2g} S_{2g}$, then $\tilde{\eta}_S(S_\delta)$ is the algebraic version of the ratio
\[\frac{\theta\begin{bsmallmatrix}{}^t(\frac{a_1}{2},\cdots,\frac{a_g}{2})\\ {}^t(\frac{a_{g+1}}{2},\cdots,\frac{a_{2g}}{2})\end{bsmallmatrix}(0,\Omega)^2}{\theta\begin{bsmallmatrix}0\\ 0\end{bsmallmatrix}(0,\Omega)^2}\]
for $k=\C$.
We claim that $\bar{S}_\delta=\psi(S_\delta)$ is the point determined by the symplectic basis $\bar{S}_1,\ldots, \bar{S}_{2g}$ with similar property for the normal Weil set $\{\beta_{\bar{S}}g_{\bar{S}}\}_{S\in J_\mathcal{C}[2]}$.
Using the notation of \cite[Page 304]{EqnAVI}, it is enough to show that
\[e_*^{\mathcal{O}(\mathcal{W})}(S)=e_*^{\mathcal{O}(\ell \mathcal{W})}(S)=e_*^{\mathcal{O}(X)}(S)=e_*^{\mathcal{O}(Y)}(\bar{S})\]
for all $S\in J_\mathcal{C}[2]$.
Indeed, this follows from \cite[Proposition 2, Page 307]{EqnAVI} and \cite[Page 304]{EqnAVI}.

To compute $S_\delta$, we note that every theta characteristics of $\mathcal{C}$ can be represented by a divisor of form
\[\theta_T=\sum_{P\in T}(P)-(g-1-\#T)(\infty),\]
where $T$ is a set of Weierstrass points with $\#T\equiv g+1\mod 2$, and that $[\theta_T]$ is even if and only if $\#T\equiv g+1\mod 4$ \cite[Proposition 6.1]{Mumford-Thomae}.
So if we write $S_\delta=\epsilon_1S_1+\cdots+\epsilon_{2g}S_{2g}$ and $S_i=[\theta_{T_i}-(g-1)(\infty)]$ for $i=1,2,\ldots,2g$, then
\[\frac{\#T_i-(g+1)}{2}\equiv \epsilon_{g+i}+\epsilon_1\epsilon_{g+1}+\cdots+\epsilon_g\epsilon_{2g}\mod 2,\]
where $\epsilon_{g+i}=\epsilon_{i-g}$ for $i=g+1,\ldots,2g$. This gives $S_\delta$.

So far we have obtained the values
\begin{equation}\label{WeilfunctiononCodomain}
  \beta_{\bar{S}}g_{\bar{S}}(\bar{S}_\delta)=\beta_{\bar{S}}\eta_{S}^\ell(S_\delta+v_0)\tilde{\theta}^2(S_\delta+v_0+S)/\tilde{\theta}^2(S_\delta+v_0)(v_0\in V),
\end{equation}
which over $\C$ correspond to
\begin{equation*}
  (\frac{\theta\begin{bsmallmatrix}{}^t(\frac{a_1}{2},\cdots,\frac{a_g}{2})\\ {}^t(\frac{a_{g+1}}{2},\cdots,\frac{a_{2g}}{2})\end{bsmallmatrix}(0,\Omega')}{\theta\begin{bsmallmatrix}0\\ 0\end{bsmallmatrix}(0,\Omega')})^2
\end{equation*}
with $S=a_1 S_1+\cdots+a_{2g} S_{2g}$.
Now assume that $\text{dim}\,\mathcal{A}=3$, and that $\mathcal{D}$ is non-hyperelliptic (this means that $\beta_{\bar{S}}g_{\bar{S}}(\bar{S}_\delta)\ne 0$ for all $a_1,\ldots,a_6\in\{0,1\}$ with $a_1a_4+a_2a_5+a_3a_6$ even). We follow Milio \cite{EneaMilio}
in finding an equation of $\mathcal{D}$. The approach is based on \cite{Weber,Riemann}.
Note that the canonical embedding of $\mathcal{D}$ is a non-singular quartic in $\mathbb{P}^2$.
Such a plane quartic is uniquely determined by an Aronhold system of bitangents which is a set of $7$ bitangents with the property that the intersection points of three arbitrary bitangents in the set with the quartic do not lie on a conic. Riemann \cite{Riemann} gave a construction of the plane quartic from an Aronhold system.

There are $288$ Aronhold systems for a given plane quartic and we use the same one as in \cite{EneaMilio,Weber-formula}.
For $i=0,1,\ldots,63$, we set
\[\vartheta_i=\theta\begin{bsmallmatrix}{}^t\!(\frac{c_3}{2},\frac{c_4}{2},\frac{c_5}{2})\\{}^t\!(\frac{c_0}{2},\frac{c_1}{2},\frac{c_2}{2})\end{bsmallmatrix}(0,\Omega'),\]
where the $c_j$ are integers in $\{0,1\}$ such that $i=\sum_{j=0}^5c_j2^j$. Then
\begin{equation}\label{Aronhold_system}
  \begin{gathered}
  x=0, y=0,z=0,x+y+z=0,\\
  \alpha_{i1}x+\alpha_{i2}y+\alpha_{i3}z=0\, (i=1,2,3)
\end{gathered}
\end{equation}
with
\begin{align*}
  \alpha_{11}&=\frac{\vartheta_{5}\vartheta_{12}}{\vartheta_{33}\vartheta_{40}},&
  \alpha_{12}&=\frac{\vartheta_{21}\vartheta_{28}}{\vartheta_{49}\vartheta_{56}},&
  \alpha_{13}&=\frac{\vartheta_{7}\vartheta_{14}}{\vartheta_{35}\vartheta_{42}},\\
  \alpha_{21}&=\frac{\vartheta_{5}\vartheta_{27}}{\vartheta_{40}\vartheta_{54}},&
  \alpha_{22}&=\frac{\vartheta_{2}\vartheta_{28}}{\vartheta_{47}\vartheta_{49}},&
  \alpha_{23}&=\frac{\vartheta_{14}\vartheta_{16}}{\vartheta_{35}\vartheta_{61}},\\
  \alpha_{31}&=-\frac{\vartheta_{12}\vartheta_{27}}{\vartheta_{33}\vartheta_{54}},&
  \alpha_{32}&=\frac{\vartheta_{2}\vartheta_{21}}{\vartheta_{47}\vartheta_{56}},&
  \alpha_{33}&=\frac{\vartheta_{7}\vartheta_{16}}{\vartheta_{42}\vartheta_{61}}.
\end{align*}
is an Aronhold system for $\mathcal{D}$.

So we have to compute the $\alpha_{ij}$. The approach was explained in \cite{EneaMilio}. First note that the values $\vartheta_{i}^2/\vartheta_{0}^2$ are the $\beta_{\bar{S}}g_{\bar{S}}(\bar{S}_\delta)$. In particular, we can obtain $\alpha_{11}^2$, $\alpha_{21}^2$ and $\alpha_{31}^2$.
Consider the following equations (they are taken from \cite{EneaMilio} and computed from the Riemann's theta formula)
\begin{align*}
\vartheta_{5}\vartheta_{12}\vartheta_{33}\vartheta_{40}-\vartheta_{21}\vartheta_{28}\vartheta_{49}\vartheta_{56}-\vartheta_{42}\vartheta_{35}\vartheta_{14}\vartheta_{7}=0,\\
\vartheta_{49}\vartheta_{47}\vartheta_{28}\vartheta_{2}-\vartheta_{54}\vartheta_{40}\vartheta_{27}\vartheta_{5}-\vartheta_{61}\vartheta_{35}\vartheta_{16}\vartheta_{14}=0,\\
\vartheta_{54}\vartheta_{33}\vartheta_{27}\vartheta_{12}-\vartheta_{56}\vartheta_{47}\vartheta_{21}\vartheta_{2}+\vartheta_{61}\vartheta_{42}\vartheta_{16}\vartheta_{7}=0.
\end{align*}
From the first one, we have
\[(\vartheta_{42}\vartheta_{35}\vartheta_{14}\vartheta_{7})^2=(\vartheta_{5}\vartheta_{12}\vartheta_{33}\vartheta_{40})^2+(\vartheta_{21}\vartheta_{28}\vartheta_{49}\vartheta_{56})^2
-2\vartheta_{5}\vartheta_{12}\vartheta_{21}\vartheta_{28}\vartheta_{33}\vartheta_{40}\vartheta_{49}\vartheta_{56}\]
and
\[(\vartheta_{21}\vartheta_{28}\vartheta_{49}\vartheta_{56})^2=(\vartheta_{5}\vartheta_{12}\vartheta_{33}\vartheta_{40})^2+(\vartheta_{42}\vartheta_{35}\vartheta_{14}\vartheta_{7})^2
-2\vartheta_{5}\vartheta_{7}\vartheta_{12}\vartheta_{14}\vartheta_{33}\vartheta_{35}\vartheta_{40}\vartheta_{42},\]
from which we can obtain $\alpha_{11}\alpha_{12}$ and $\alpha_{11}\alpha_{13}$. Hence if we choose an arbitrary root $\epsilon\alpha_{11}$ of $\alpha_{11}^2$ ($\epsilon\in\{\pm1\}$), then dividing $\alpha_{11}\alpha_{12}$ and $\alpha_{11}\alpha_{13}$ respectively by $\epsilon\alpha_{11}$, we get $\epsilon\alpha_{12}$ and $\epsilon\alpha_{13}$, and so the bitangent $\epsilon(\alpha_{11}x+\alpha_{12}y+\alpha_{13}z)=0$ is obtained. Using the second and the third equations we see that the same argument gives the bitangents $\alpha_{i1}x+\alpha_{i2}y+\alpha_{i3}z=0$ for $i=2,3$.

Using the above notation, we have \cite[Proposition 3]{Ritzenthaler}:
\begin{theorem}[Riemann]
The curve $\mathcal{D}$ has an equation
\[(x\xi_1+y\xi_2-z\xi_3)^2=4xy\xi_1\xi_2,\]
where $\xi_1,\xi_2,\xi_3$ are linear functions of $x,y,z$ defined by
\[\left(
    \begin{matrix}
      1 & 1 & 1 \\
      \frac{1}{\alpha_{11}} & \frac{1}{\alpha_{12}} & \frac{1}{\alpha_{13}} \\
     \frac{1}{\alpha_{21}} & \frac{1}{\alpha_{22}} & \frac{1}{\alpha_{23}} \\
      \frac{1}{\alpha_{31}} & \frac{1}{\alpha_{32}} & \frac{1}{\alpha_{33}}\\
    \end{matrix}
  \right)\left(
           \begin{matrix}
             \xi_1 \\
             \xi_2 \\
             \xi_3\\
           \end{matrix}
         \right)=-\left(
    \begin{matrix}
      1 & 1 & 1 \\
      k_1\alpha_{11} & k_1\alpha_{12} & k_1\alpha_{13} \\
     k_2\alpha_{21} & k_2\alpha_{22} & k_2\alpha_{23} \\
      k_3\alpha_{31} & k_3\alpha_{32} & k_3\alpha_{33}\\
    \end{matrix}
  \right)\left(
           \begin{matrix}
             x \\
             y \\
             z\\
           \end{matrix}
         \right)
\]
with $k_1, k_2, k_3$ solutions of
\[\left(\begin{matrix}
       \frac{1}{\alpha_{11}} & \frac{1}{\alpha_{21}} & \frac{1}{\alpha_{31}} \\
     \frac{1}{\alpha_{12}} & \frac{1}{\alpha_{22}} & \frac{1}{\alpha_{32}} \\
      \frac{1}{\alpha_{13}} & \frac{1}{\alpha_{23}} & \frac{1}{\alpha_{33}}\\
    \end{matrix}\right)
  \left(
           \begin{matrix}
             u_1 \\
             u_2 \\
             u_3 \\
           \end{matrix}
         \right)\!=\!\left(
                   \begin{matrix}
                     -1 \\
                     -1 \\
                    -1 \\
                   \end{matrix}
                 \right),
                 \left(
                   \begin{matrix}
                      u_1\alpha_{11} & u_2\alpha_{21} & u_3\alpha_{31} \\
                      u_1\alpha_{12} & u_2\alpha_{22} & u_3\alpha_{32} \\
                      u_1\alpha_{13} & u_2\alpha_{23} & u_3\alpha_{33}\\
                   \end{matrix}
                 \right)\left(
                          \begin{matrix}
                            k_1 \\
                            k_2 \\
                            k_3 \\
                          \end{matrix}
                        \right)\!=\!\left(
                                  \begin{matrix}
                                    -1 \\
                                    -1 \\
                                    -1 \\
                                  \end{matrix}
                                \right).
\]
\end{theorem}

\subsection{Image point}\label{image_computation}
For the moment $\psi$ is not completely determined, it is determined by the kernel $V$ up to an automorphism of the polarized abelian variety $(\mathcal{A},\mathcal{O}(Y))$.
Assume that the automorphism group is $\{\pm1\}$. Then we can fix a point $P\in J_\mathcal{C}$ and specify a point
$R\in\{R',R''\}=\{\psi(P),-\psi(P)\}$ to distinguish $\psi$ from $-\psi$, for example, $\psi(P)=R$. In this subsection, we discuss how to
compute $R',R''$.

The idea of the computation is this: we first compute the Schr\"odinger coordinates of $\kappa(\psi(P))$ and $\kappa(0)$ (by abuse of notation we will denote by $\kappa$ the Kummer maps $\mathcal{A}\to \mathbb{P}^7$ and $J_\mathcal{D}\to \mathbb{P}^7$).
Then we compute the linear transformation of $\mathbb{P}^7$ which allows us to identify $\kappa(\mathcal{A})$ with $\kappa(J_\mathcal{D})$.
 Next, we construct eight "good" functions on $\mathcal{D}^3$ which can give an embedding of the Kummer variety of $J_\mathcal{D}$ in $\mathbb{P}^7$,
and write each as a linear combination of the normalised Weil functions. Then we reduce to computing a point in $\mathcal{D}^3$ at which
the eight "good" functions have given values (up to a constant), and we obtain this point by computing a Gr\"obner basis.

%\[\xymatrix@C=2cm{\mathcal{A}\ar[r]& \mathbb{P}^7\ar[d]^\alpha\\
%J_\mathcal{D}\ar[r]^{(\eta_1:\cdots:\eta_8)}&\mathbb{P}^7}
%\]
\subsubsection{Computation of the Schr\"odinger coordinates}
Using the notations of subsection \ref{subsec: equation_of_D}, we have functions $\xi_{\bar{S}}=T_{\bar{S}_\delta}^*\beta_{\bar{S}}g_{\bar{S}}\in H^0(\mathcal{A},\mathcal{O}(2T_{\bar{S}_\delta}^*Y))$
which can be thought of as
\begin{equation*}
  (\frac{\theta\begin{bsmallmatrix}{}^t(\frac{i_1}{2}&\frac{i_2}{2}&\frac{i_3}{2})\\ {}^t(\frac{i_4}{2}&\frac{i_5}{2}&\frac{i_6}{2})\end{bsmallmatrix}(z,\Omega')}{\theta\begin{bsmallmatrix}{}^t(0&0&0)\\ {}^t(0&0&0)\end{bsmallmatrix}(z,\Omega')})^2,
\end{equation*}
where $i_1,i_2,\ldots,i_6\in\{0,1\}$ satisfy $S=i_1 S_1+\cdots+i_6 S_6$. So we can compute the image of $\psi(P)$ under $\kappa:\mathcal{A}\to \mathbb{P}^7$
by using Proposition \ref{to_Schroedinger}.

Next we study the computation for $\kappa:J_\mathcal{D}\to \mathbb{P}^7$.

It is easy to obtain a symplectic basis of $J_\mathcal{D}[2]$ from the Aronhold system \eqref{Aronhold_system} for $\mathcal{D}$.
Let $L_1,\ldots,L_7$ be the odd theta characteristic corresponding to the bitangents in \eqref{Aronhold_system}. Then there exist an even theta characteristic $\delta_\mathcal{D}$ and a symplectic basis $E_1,\ldots,E_6$ such that
\begin{equation}\label{theta_character_symplectic_basis}
(L_1-\delta_\mathcal{D},\ldots, L_7-\delta_\mathcal{D})=(E_1,\ldots,E_6)\begin{bmatrix}
1&0&0&1&1&1&0\\
1&0&1&0&0&1&1\\
1&1&1&1&0&0&0\\
1&0&0&1&1&0&1\\
1&1&0&0&0&1&1\\
1&1&1&0&1&0&0
\end{bmatrix}.
\end{equation}
Here the column vectors of the $(0,1)$-matrix form an Aronhold system in terms of reduced characteristics (see \cite[Page 10]{Weber-formula}).
So we can express $E_1,\ldots,E_6$ and $\delta_D$ in terms of $L_1,\ldots, L_7$ (the reason for using this basis will appear below).

Now let $d''$ be the bilinear pairing on $J_\mathcal{D}[2]$ which takes $-1$ on the pairs $(E_i,E_j)$ exactly for $(i,j)=(1,4),(2,5),(3,6)$.
Take $\mathcal{W}'=\{[(P_1)+(P_2)]-\delta_D: P_1,P_2\in \mathcal{D}\}$. Then we can compute a normal Weil set $\{\zeta_P\}_{P\in J_\mathcal{D}[2]}$ with respect to $d''$ such that $\text{div}(\zeta_P)=2T_P^*\mathcal{W}'-2\mathcal{W}'$. These $\zeta_P$ can be thought of as
\begin{equation*}
  (\frac{\theta\begin{bsmallmatrix}{}^t(\frac{i_1}{2}&\frac{i_2}{2}&\frac{i_3}{2})\\ {}^t(\frac{i_4}{2}&\frac{i_5}{2}&\frac{i_6}{2})\end{bsmallmatrix}(z,\Omega'')}{\theta\begin{bsmallmatrix}{}^t(0&0&0)\\ {}^t(0&0&0)\end{bsmallmatrix}(z,\Omega'')})^2
\end{equation*}
with $i_1,i_2,\ldots,i_6\in\{0,1\}$ satisfying $P=i_1 E_1+\cdots+i_6 E_6$.

\subsubsection{Computation of the linear transformation }
By abuse of language, we will say that the image of a two-torsion point in the Kummer variety "is a two-torsion point".

First, we compute the Schr\"odinger coordinates of the identity elements of $\mathcal{A}$ and $J_\mathcal{D}$ respectively. From this, as mentioned in section \ref{kummer_variety_intro}, it is easy to obtain
the Schr\"odinger coordinates of all two-torsion points. Then we look for the linear change of coordinates which sends the two-torsion points in $\kappa(\mathcal{A})$ to the two-torsion points in $\kappa(J_\mathcal{D})$ and
preserves the group structure and the Weil pairing (this is because the principally polarized abelian varieties $(A,\mathcal{O}(Y))$ and $(J_\mathcal{D},\mathcal{O}(\mathcal{W}'))$ are isomorphic). Using four nested for loops, we obtain $8+15\times 7=113$ linear equations which are enough for computation. In our experiments, it took around half an hour to examine all the possibilities, but a solution was found in a few minutes. This is similar to the cost in \cite{EneaMilio}.

The process of computing the linear change of coordinates is essentially to determine theta structure.
Thus this method can be improved after observing that $\bar{S}_i$ can be identified with $E_i$.
In fact, two quartics with the same set of bitangents coincide, and the bases $\bar{S}_1,\ldots, \bar{S}_6$ and $E_1,\ldots,E_6$ are determined by $L_1,\ldots,L_7$
and the $(0,1)$-matrix in (\ref{theta_character_symplectic_basis}).
Now, let $(\bar{S}_1,\ldots,\bar{S}_6)=(E_1,\ldots,E_6)$. Then as we explained in subsection \ref{sec:normalization}, the normalised Weil functions $\zeta_P$ and $\xi_P$ differ by a character $\chi$ of $J_\mathcal{D}[2]$, and this character $\chi$ can be determined by comparing the values of $\zeta_P(0)$ and $\xi_P(0)$ for $P=E_1,\ldots, E_6$.

\subsubsection{Computation of a lift}
This part is based on a discussion with J.-M. Couveinges.
Now let $K_\mathcal{D}=2(\infty_1)+2(\infty_2)$ be the divisor cut by the bitangent $z=0$. We consider the morphism
\begin{equation*}
  j: \mathcal{D}^3\to J_\mathcal{D}, (P_1,P_2,P_3)\mapsto [(P_1)+(P_2)+(P_3)-(\infty_1)-\delta_D].
\end{equation*}
The pullback of $\mathcal{W}'=\{[(P_1)+(P_2)]-\delta_D: P_1,P_2\in \mathcal{D}\}$ is
\[D_0+\sum_{i=1}^3pr_i^*((\infty_1))\]
with $D_0=\{(P_1,P_2,P_3)\in \mathcal{D}^3|h^0(\mathcal{D},\mathcal{O}((P_1)+(P_2)+(P_3)))=2 \}$.
Let $x_1, y_1$, $x_2, y_2$, $x_3, y_3$ be the affine coordinate functions on the factors of $\mathcal{D}^3$. Then
function $\omega=\det\begin{bsmallmatrix}
      x_1 & y_1 & 1 \\
      x_2 & y_2 & 1 \\
      x_3 & y_3 & 1 \\
    \end{bsmallmatrix}
$ has divisor
\[\text{div}(\omega)=D_0+\Delta-\sum_{i=1}^3pr_i^*(K_\mathcal{D}),\]
where $\Delta$ is the full diagonal. Note that if $f\in H^0(J_\mathcal{D},\mathcal{O}(2\mathcal{W}'))$, then
\[\text{div}(\omega^2j^*f)+\sum_{i=1}^3pr_i^*(2(\infty_1)+2K_\mathcal{D})\geq 2\Delta.\]
So if we find a basis $h_1,\ldots,h_8$ of the subspace of
\begin{equation*}
  H^0(\mathcal{D}^3,\mathcal{O}(\sum_{i=1}^3pr_i^*(2(\infty_1)+2K_\mathcal{D})))=\otimes_{i=1}^3pr_i^*H^0(\mathcal{D},\mathcal{O}(2(\infty_1)+2K_\mathcal{D}))
\end{equation*}
corresponding to functions which are invariant under permutations of factors of $\mathcal{D}^{3}$ and vanish along $\Delta$, then
we can represent $j^*f$ as a linear combination of $h_1/\omega^2,\ldots, h_8/\omega^2$.

The space $H^0(\mathcal{D},\mathcal{O}(2(\infty_1)+2K_\mathcal{D}))$ has a basis $\mathfrak{B}=\{1,x,x^2,y,y^2,xy,u,v\}$, where $u,v$ are polynomials in $x,y$ of degree $3$.
So $h_1,\ldots, h_8$ can be expressed linearly in terms of the functions of form
\begin{equation}\label{symmetric}
  \sum_{\sigma\in \text{Sym}(\{1,2,3\})}pr_{\sigma(1)}^*t_1 \cdot pr_{\sigma(2)}^*t_2\cdot  pr_{\sigma(3)}^*t_3 \,\,(t_1,t_2,t_3\in\mathfrak{B}).
\end{equation}
The number of functions in \eqref{symmetric} is $8\times 7\times 6+8\times 7+8=120$. We pick $I$ random affine points in $\Delta$, compute the values of these functions at each point to get
an $I\times 120$ matrix. Then we compute the row echelon form of this matrix. If the rank of this matrix is $112$, then we deduce $8$ linearly independent functions $h_1,\ldots,h_8$.

Next, we have to express each of the $h_i/\omega^2$ as a combination of the $j^*\zeta_P$. Using proposition \ref{schroedinger_to}, we can easily compute a basis of the $\zeta_P$ from the Schr\"odinger coordinates of the identity element. So in general after evaluating these $16$ functions at $9$ points, we can find the expressions.

Let $(R_1,R_2,R_3)\in \mathcal{D}^3$ be a point such that $\kappa(j(R_1,R_2,R_3))=\tau(\kappa(\psi(P)))$, where $\tau$ is the linear change of the Schr\"odinger
coordinates. So now we know the projective coordinates $[\lambda_1:\cdots:\lambda_8]$ for the point $[h_1(R_1,R_2,R_3):\cdots:h_8(R_1,R_2,R_3)]$.
Assume that $R_i=[x_i:y_i:1]$ for $i=1,2,3$. Then we can compute the Gr\"obner basis of the ideal in $k[x_1,y_1,x_2,y_2,x_3,y_3,\lambda]$ generated by the following $11$ polynomials:
\begin{gather*}
 h(R_i) (i=1,2,3),\,  \lambda h_{i_0}(R_1,R_2,R_3)-1, \\
  h_1(R_1,R_2,R_3)\lambda_i-h_i(R_1,R_2,R_3)\lambda_1(i=2,3,\ldots,8),
\end{gather*}
where $h$ is the affine equation of $\mathcal{D}$, and $h_{i_0}$ is assumed to be non-zero at $(R_1,R_2,R_3)$. In practice, the computation is easy,
and we find a solution.

\begin{remark}
For translating the discrete logarithm problem from $J_\mathcal{C}(\F_q)$ to $J_\mathcal{D}(\F_q)$ the results in this subsection are sufficient. In fact, for $P,Q\in J_\mathcal{C}(\F_q)$ with $Q=mP$, we can compute $R,R',R''\in J_\mathcal{D}(\F_q)$ such that $\psi(P)=R,\psi(Q)\in\{R',R''\}$. Then we compute the discrete logarithm $m'$ of $R'$ with respect to $R$. If $m'P=Q$, then $m=m'$, otherwise $m=-m'$. Here of course we assume the abelian variety $J_\mathcal{C}/V$ is isomorphic to $J_\mathcal{D}$ over $\F_q$ and
the intersection of the subgroup $\langle P\rangle$ with the kernel $V$ is trivial.
\end{remark}
\subsection{Equation of isogeny}
Let $F:\mathcal{C}\to J_\mathcal{D}$ be the morphism given by $F((x,y))=\psi([((x,y))-(\infty)])$. We first give an algebraic
description of $F$ which is slightly different from that of \cite{Couveignes-Ezome,EneaMilio}. Then we show how to compute it.

For $P=(x_0,y_0)\in \mathcal{C}$, we have points $R_1,R_2,R_3\in \mathcal{D}$ such that
\[F(P)=[(R_1)+(R_2)+(R_3)-(\infty_1)-\delta_\mathcal{D}].\]
The divisor $(R_1)+(R_2)+(R_3)$ is unique in general. To describe it, we use the following functions:
\begin{eqnarray*}
   \mathbf{p}_1&=& x(R_1)+x(R_2)+x(R_3),\\
   \mathbf{p}_2&=& x(R_1)x(R_2)+x(R_1)x(R_3) +x(R_2)x(R_3),\\
   \mathbf{p}_3&=& x(R_1)x(R_2)x(R_3),\\
   \mathbf{p}_4&=& y(R_1)+y(R_2)+y(R_3),\\
   \mathbf{p}_5&=& y(R_1)y(R_2)+y(R_1)y(R_3) +y(R_2)y(R_3),\\
   \mathbf{p}_6&=& y(R_1)y(R_2)y(R_3).\\
\end{eqnarray*}
So there are rational functions (in one variable) $\mathbf{A}_i,\mathbf{B}_i$ such that $\mathbf{p}_i((R_1)+(R_2)+(R_3))=\mathbf{A}_i(x_0)+\mathbf{B}_i(x_0)y_0$ for $i=1,\ldots,6$.
Let $\bar{P}=(x_0,-y_0)$. Then
\[F(\bar{P})=-F(P)=[(\bar{R}_1)+(\bar{R}_2)+(\bar{R}_3)-(\infty_1)-\delta_\mathcal{D}],\]
and $\mathbf{p}_i((\bar{R}_1)+(\bar{R}_2)+(\bar{R}_3))=\mathbf{A}_i(x_0)-\mathbf{B}_i(x_0)y_0$ for $i=1,\ldots,6$.

The map $F$ induces a map $F^*: H^0(J_\mathcal{D},\Omega_{J_\mathcal{D}}^1)\to H^0(\mathcal{C},\Omega_{\mathcal{C}}^1)$
of vector spaces of regular differentials. It is well known that $H^0(J_\mathcal{D},\Omega_{J_\mathcal{D}}^1)$ is isomorphic to the invariant subspace of $\oplus_{i=1}^3pr_i^*H^0(\mathcal{D},\Omega_{\mathcal{D}}^1)$ by permutations of
the three factors. Let $h$ be an affine equation of $\mathcal{D}$. Then $\omega_1=\frac{dx_1}{h_{y}(x_1,y_1)}+\frac{dx_2}{h_{y}(x_2,y_2)}+\frac{dx_3}{h_{y}(x_3,y_3)}$,
$ \omega_2=\frac{x_1dx_1}{h_{y}(x_1,y_1)}+\frac{x_2dx_2}{h_{y}(x_2,y_2)}+\frac{x_3dx_3}{h_{y}(x_3,y_3)}$,
$\omega_3=\frac{y_1dx_1}{h_{y}(x_1,y_1)}+\frac{y_2dx_2}{h_{y}(x_2,y_2)}+\frac{y_3dx_3}{h_{y}(x_3,y_3)}$
%\begin{gather*}
%  \omega_1=\frac{dx_1}{h_{y}(x_1,y_1)}+\frac{dx_2}{h_{y}(x_2,y_2)}+\frac{dx_3}{h_{y}(x_3,y_3)}, \\
%  \omega_2=\frac{x_1dx_1}{h_{y}(x_1,y_1)}+\frac{x_2dx_2}{h_{y}(x_2,y_2)}+\frac{x_3dx_3}{h_{y}(x_3,y_3)},\\
%  \omega_3=\frac{y_1dx_1}{h_{y}(x_1,y_1)}+\frac{y_2dx_2}{h_{y}(x_2,y_2)}+\frac{y_3dx_3}{h_{y}(x_3,y_3)}
%\end{gather*}
is a basis of this space, where $h_y(x,y)=\partial h(x,y)/\partial y$. We can represent $F^*\omega_1,F^*\omega_2,F^*\omega_3$ with respect to the basis $dx/y,xdx/y,x^2dx/y$ of $H^0(\mathcal{C},\Omega_{\mathcal{C}}^1)$:
\[F^*\omega_i=(m_{i1}+m_{i2}x+m_{i3}x^2)dx/y,\quad i=1,2,3.\]

Fix a point $P=(x_0,y_0)\in \mathcal{C}$ with $y_0\ne 0$. Let $t=x-x_0$ be the local parameter at $P$. Then we have a formal point $P(t)=(U(t), V(t))$ with $U(t)=x_0+t$ and $V(0)=y_0$.
Let $R_i(t)=(X_i(t),Y_i(t))$, $\bar{R}_i(t)=(\bar{X}_i(t),\bar{Y}_i(t))$, $i=1,2,3$ be formal points on $\mathcal{D}$ such that
\begin{eqnarray*}
 F(P(t))&=&[(R_1(t))+(R_2(t))+(R_3(t))-(\infty_1)-\delta_\mathcal{D}],  \\
 F(\bar{P}(t))&=&[(\bar{R}_1(t))+(\bar{R}_2(t))+(\bar{R}_3(t))-(\infty_1)-\delta_\mathcal{D}]
\end{eqnarray*}
with $\bar{P}(t)=(U(t), -V(t))$.
Then we have two systems of differential equations
\begin{equation}\label{differential_eqn1}
  \left\{ \begin{aligned}
   \frac{\dot{X}_1(t)}{h_{y}(X_1,Y_1)}+\frac{\dot{X}_2(t)}{h_{y}(X_2,Y_2)}+\frac{\dot{X}_3(t)}{h_{y}(X_3,Y_3)}&=\frac{(m_{11}+m_{12}U+m_{i3}U^2)\dot{U}(t)}{V},\\
   \frac{X_1\dot{X}_1(t)}{h_{y}(X_1,Y_1)}+\frac{X_2\dot{X}_2(t)}{h_{y}(X_2,Y_2)}+\frac{X_3\dot{X}_3(t)}{h_{y}(X_3,Y_3)}&=\frac{(m_{21}+m_{22}U+m_{23}U^2)\dot{U}(t)}{V},\\
   \frac{Y_1\dot{X}_1(t)}{h_{y}(X_1,Y_1)}+\frac{Y_2\dot{X}_2(t)}{h_{y}(X_2,Y_2)}+\frac{Y_3\dot{X}_3(t)}{h_{y}(X_3,Y_3)}&=\frac{(m_{31}+m_{32}U+m_{33}U^2)\dot{U}(t)}{V},\\
   h(X_1(t),Y_1(t))&=0,\\
   h(X_2(t),Y_2(t))&=0,\\
   h(X_3(t),Y_3(t))&=0.
   \end{aligned}\right.
\end{equation}
and
\begin{equation}\label{differential_eqn2}
  \left\{ \begin{aligned}
   \frac{\dot{\bar{X}}_1(t)}{h_{y}(\bar{X}_1,\bar{Y}_1)}+\frac{\dot{\bar{X}}_2(t)}{h_{y}(\bar{X}_2,\bar{Y}_2)}+\frac{\dot{\bar{X}}_3(t)}{h_{y}(\bar{X}_3,\bar{Y}_3)}&=\frac{(\bar{m}_{11}+\bar{m}_{12}U+\bar{m}_{i3}U^2)\dot{U}(t)}{V},\\
   \frac{\bar{X}_1\dot{\bar{X}}_1(t)}{h_{y}(\bar{X}_1,\bar{Y}_1)}+\frac{\bar{X}_2\dot{\bar{X}}_2(t)}{h_{y}(\bar{X}_2,\bar{Y}_2)}+\frac{\bar{X}_3\dot{\bar{X}}_3(t)}{h_{y}(\bar{X}_3,\bar{Y}_3)}&=\frac{(m_{21}+m_{22}U+m_{23}U^2)\dot{U}(t)}{V},\\
   \frac{\bar{Y}_1\dot{\bar{X}}_1(t)}{h_{y}(\bar{X}_1,\bar{Y}_1)}+\frac{\bar{Y}_2\dot{\bar{X}}_2(t)}{h_{y}(\bar{X}_2,\bar{Y}_2)}+\frac{\bar{Y}_3\dot{\bar{X}}_3(t)}{h_{y}(\bar{X}_3,\bar{Y}_3)}&=\frac{(m_{31}+m_{32}U+m_{33}U^2)\dot{U}(t)}{V},\\
   h(\bar{X}_1(t),\bar{Y}_1(t))&=0,\\
   h(\bar{X}_2(t),\bar{Y}_2(t))&=0,\\
   h(\bar{X}_3(t),\bar{Y}_3(t))&=0.
   \end{aligned}\right.
\end{equation}

Now the rational functions $\mathbf{A}_1,\mathbf{B}_1,\ldots,\mathbf{A}_6,\mathbf{B}_6$ can be computed as follows.
\begin{enumerate}
  \item Compute $R_1(t),R_2(t),R_3(t)$ at precision $4$ with the method described in subsection \ref{image_computation}. Then comparing the coefficients of $t^0,t^1,t^2$ in \eqref{differential_eqn1}, we get
  $9$ linear equations with the $9$ unknown $m_{11},m_{12},\ldots, m_{33}$. We can solve this system and get the values $m_{ij}$ for $i,j\in\{1,2,3\}$.
  \item Increase the accuracy of $X_1(t),Y_1(t),X_2(t),Y_2(t),X_3(t), Y_3(t)$. Their coefficients can be computed one by one using (\ref{differential_eqn1}). See \cite{Couveignes-Ezome,EneaMilio} for more details.
  \item Repeat the above process for $\bar{X}_1(t), \ldots, \bar{Y}_3(t)$ by using \eqref{differential_eqn2}.
  \item Using continued fraction to recover $\mathbf{A}_i,\mathbf{B}_i$ from
  \begin{eqnarray*}
  \mathbf{A}_i(U(t))&=&\frac{\mathbf{p}_i(R_1(t),R_2(t),R_3(t))+\mathbf{p}_i(\bar{R}_1(t),\bar{R}_2(t),\bar{R}_3(t))}{2} \\
   \mathbf{B}_i(U(t))&=&\frac{\mathbf{p}_i(R_1(t),R_2(t),R_3(t))-\mathbf{p}_i(\bar{R}_1(t),\bar{R}_2(t),\bar{R}_3(t))}{2V(t)}
  \end{eqnarray*}
 for $i=1,2,\ldots,6$.
\end{enumerate}
\begin{remark}
If we like, we can write $F(P)$ as $[(R_1')+(R_2')+(R_3')-3(O')]$ with $O'$ a fixed rational point on $\mathcal{D}$.
Once the expressions for $\mathbf{p}_1,\ldots,\mathbf{p}_6$ is given, we can use interpolation to find the rational fractions for the six functions
$x(R_1')+x(R_2')+x(R_3')$, $x(R_1')x(R_2')+x(R_1')x(R_3')+x(R_2')x(R_3')$, $x(R_1')x(R_2')x(R_3')$, $y(R_1')+y(R_2')+y(R_3')$,$y(R_1')y(R_2')+y(R_1')y(R_3')+y(R_2')y(R_3')$,$y(R_1')y(R_2')y(R_3')$.

\end{remark}

\section{Evaluation of Weil functions}
The main aspect of computation is to evaluate Weil functions, which reduces to evaluate functions like \eqref{WeilFunction}
\begin{equation*}
  g_P=(\frac{\det(f_i^{D_P}(z_j))}{\det(f_i^{D_0}(z_j))})^N\prod_{i=1}^{g}h^{E_P}(z_i).
\end{equation*}
To compute the Weil pairing for the $\ell$-torsion points or a normalization of a Weil set (see \eqref{WeilPairing},\eqref{Npower},\eqref{normalization}), we can always
choose pairwise distinct points $z_1,\ldots,z_g$ such that the $f_i^{D_P}$, $f_i^{D_0}$, $h^{E_P}$ are regular at these points and
that the denominator is nonzero. However, we also need to evaluate functions at some given point $(z_1,\ldots,z_g)$ which may not satisfy these
conditions. This difficulty is resolved by Couveignes and Ezome in \cite{Couveignes-Ezome}. Assume for example that $z_1=z_2=\cdots=z_g$.
We fix a local parameter $u$ at $z_1$ and $g$ pairwise distinct scalars $c_1,\ldots,c_g$ in $\mathbb{F}_q$ (or a small degree extension of it). Then we consider the field $k((t))$ of formal series in $t$ with coefficients
in $k$, and take $g$ points $z_1(t),\ldots,z_g(t)$ in $\mathcal{C}(k((t)))$ associated with the values $c_1t,\ldots,c_gt$ of the local parameter $u$.
We do the evaluation with $(z_1,\ldots,z_g)$ replaced by $(z_1(t),\ldots,z_g(t))$ and set $t=0$ in the result. It is pointed out in \cite{Couveignes-Ezome} that the necessary $t$-adic accuracy is $g(g-1)/2$.
%Note that it may happen that we evaluate Weil functions at points on their pole divisors. But this happens only in the cases that we compute projective points, so Couveignes and Ezome's method for evaluating functions with a use of formal points still works.
%

\section{An example}
We now give an example of computation which was done with the computational algebra system Magma \cite{magma}.

Let $\mathcal{C}/\mathbb{F}_{257}$ be the hyperelliptic curve given by \[y^2= x^7 + 13x^5 + 6x^4 + 138x^3 + 125x^2 + 104x + 167.\]
Let $\mathbb{F}_{257^6}=\mathbb{F}_{257}[b]$ with $b^6 + 3b^4 + 62b^3 + 18b^2 + 138b + 3=0$.
Then the following three points (with Mumford representation)
\begin{align*}
  T_1=&\quad  ( x^3+15x^2+224x+57,168x^2 + 119x + 53) \\
  T_2=& \quad ( x^3 + b^{41470257160332}x^2 + b^{ 226656780125958}x + b^{186444594999936},\\
   &\quad  b^{111992175485190}x^2 + b^{214167145454262}x + b^{155385077009526})\\
  T_3=& \quad( x^3+b^{79934907834054}x^2+ b^{95549446438992}x+ b^{131411603284416},\\
 & \quad b^{60233090689044}x^2+b^{87544041188058}x +b^{104379562339416} )
\end{align*}
generate an $\mathbb{F}_{257}$-rational maximal isotropic subgroup $V$ of $J_\mathcal{C}[3]$.
Applying Riemann's construction, we get the following equation
\begin{align*}
&x^4+b^{287508602266704}x^3y+
    b^{257542339384785}x^3z+
    b^{256360092006564}x^2y^2\\
 &   +
    b^{64696415568330}x^2yz+
    b^{199437716393916}x^2z^2+
    b^{50793011898993}xy^3\\
 &+
    b^{10112359493418}xy^2z+
    b^{253764382818219}xyz^2+
    b^{121880543368038}xz^3\\
 &+
    b^{102842434295874}y^4+
    b^{238870683374505}y^3z+
    b^{100316525690745}y^2z^2\\
    &+
    b^{195856073967645}yz^3+
    b^{16813215482154}z^4=0
\end{align*}
for $\mathcal{D}$. We compute its normalised Dixmier-Ohno Invariants and reconstruct
a new curve over $\mathbb{F}_{257}$ with same invariants \cite{Ritzenthaler_quartics}. This new curve $\mathcal{D}'$ is defined by
\begin{align*}
& x^4 + 89x^3y + 244x^2y^2 + 126xy^3 + 113y^4 + 131x^3z + 3x^2yz + 255xy^2z \\
& + 65y^3z + 172x^2z^2 +139xyz^2 + 21y^2z^2 + 201xz^3 + 228yz^3 + 70z^4.
\end{align*}
Under the isomorphism $\mathcal{D}\to \mathcal{D}'$,
the Aronhold system \eqref{Aronhold_system} of bitangents for $\mathcal{D}$ corresponds to
\begin{align*}
&x+b^{215522388288753}y+ b^{47695611018783}z,x+b^{216848599369857}y+ b^{96839584139298}z,\\
&x+b^{20246531669091}y+ b^{254174461244613}z,x+b^{38739323684880}y+ b^{58104622990869}z,\\
&x+b^{16896103674723}y+ b^{203918041329093}z,x+b^{179653613588631}y+ b^{27296390573907}z,\\
&x+b^{139396127218803}y+ b^{219732235963968}z.
\end{align*}
%We denote by $\delta_{\mathcal{D}'}$ the theta characteristic determined by (\ref{theta_character_symplectic_basis}).

Let $F$ be the composition of the embedding $C\to J_\mathcal{C}, P\mapsto [(P)-(\infty)]$ with the isogeny $J_\mathcal{C}\to J_\mathcal{C}/V=J_{\mathcal{D}'}$.
Then $F(P)$ can be represented by the class of divisor $D-3((137:1:0))$ with $D$ an effective divisor of degree $3$.
Let $P=(u,v)$. Then $D$ is the divisor on $\mathcal{D}'$ cut by the two polynomials
\[x^3-\mathbf{C}_1(u,v)x^2z+\mathbf{C}_2(u,v)xz^2-\mathbf{C}_3(u,v)z^3\]
and
\[y^3-\mathbf{C}_4(u,v)y^2z+\mathbf{C}_5(u,v)yz^2-\mathbf{C}_6(u,v)z^3,\]
where $\mathbf{C}_1,\ldots,\mathbf{C}_6$ are rational functions on $\mathcal{C}$ given by

\begin{align*}
  \mathbf{C}_1(x,y)&=[(63+142x+147x^{2}+10x^{3}+202x^{4}+122x^{5}+127x^{6}+64x^{7}+245x^{8}\\
  &+169x^{9}+249x^{10}+187x^{11}+209x^{12}+139x^{13}+91x^{14}+215x^{15}+206x^{16}\\
  &+77x^{17}+62x^{18}+85x^{19}+250x^{20}+239x^{21}+104x^{22}+57x^{23}+58x^{24}\\
  &+132x^{25}+200x^{26}+25x^{27}+73x^{28}+29x^{30}+196x^{31}+106x^{32}+14x^{33}\\
  &+83x^{34})+y(103x+207x+60x^{2}+219x^{3}+90x^{4}+15x^{5}+93x^{6}+200x^{7}\\
  &+211x^{8}+35x^{9}+142x^{10}+185x^{11}+89x^{12}+54x^{13}+95x^{14}+197x^{15}\\
  &+202x^{16}+50x^{17}+131x^{18}+198x^{19}+132x^{20}+16x^{21}+13x^{22}+76x^{23}\\
  &+110x^{24}+240x^{25}+129x^{26}+12x^{27}+39x^{28}+246x^{29}+54x^{30})]/\mathbf{p}(x),
\end{align*}
\begin{align*}
  \mathbf{C}_2(x,y)&=[(181+22x+86x^{2}+48x^{3}+200x^{4}+188x^{5}+5x^{6}+153x^{7}+223x^{8}\\
  &+228x^{9}+37x^{10}+41x^{11}+16x^{12}+217x^{13}+80x^{14}+208x^{15}+218x^{16}\\
  &+162x^{17}+167x^{18}+41x^{19}+215x^{20}+91x^{21}+117x^{22}+18x^{23}+194x^{24}\\
  &+237x^{25}+68x^{26}+49x^{27}+164x^{28}+184x^{29}+98x^{30}+86x^{31}+211x^{32}\\
  &+201x^{33}+188x^{34})+y(168+179x+79x^{2}+22x^{3}+51x^{4}+36x^{5}+123x^{6}\\
  &+43x^{7}+134x^{8}+189x^{9}+126x^{10}+188x^{11}+90x^{12}+55x^{13}+224x^{14}\\
  &+103x^{15}+32x^{16}+99x^{17}+118x^{18}+79x^{19}+208x^{20}+116x^{21}+118x^{22}\\
  &+205x^{23}+236x^{24}+8x^{25}+14x^{26}+18x^{27}+135x^{28}+157x^{29}+139x^{30}\\
  &+51x^{31})]/\mathbf{p}(x),
\end{align*}
\begin{align*}
 \mathbf{C}_3(x,y)&=[(34+42x+33x^{2}+147x^{3}+62x^{4}+245x^{5}+74x^{6}+235x^{7}+120x^{8}\\
 &+116x^{9}+199x^{10}+233x^{11}+64x^{12}+169x^{13}+204x^{14}+242x^{15}+207x^{16}\\
 &+102x^{17}+120x^{18}+256x^{19}+39x^{20}+203x^{21}+19x^{22}+86x^{23}+128x^{24}\\
 &+53x^{25}+192x^{26}+5x^{27}+166x^{28}+39x^{29}+161x^{30}+230x^{31}+178x^{32}\\
 &+122x^{33}+174x^{34}+190x^{35})+y(79+233x+180x^{2}+115x^{4}+199x^{5}\\
 &+235x^{6}+77x^{7}+198x^{8}+148x^{9}+30x^{10}+96x^{11}+166x^{12}+236x^{13}\\
 &+95x^{15}+96x^{16}+37x^{17}+247x^{18}+147x^{19}+144x^{20}+93x^{21}+246x^{22}\\
 &+80x^{23}+178x^{24}+113x^{25}+213x^{26}+240x^{27}+93x^{28}+234x^{29}+211x^{30}\\
 &+87x^{31}
)]/\mathbf{p}(x),
\end{align*}
\begin{align*}
\mathbf{C}_4(x,y)&=[(116+173x+41x^{2}+5x^{3}+85x^{4}+22x^{5}+203x^{6}+239x^{7}+126x^{8}\\
&+25x^{9}+102x^{10}+110x^{11}+180x^{12}+252x^{13}+104x^{14}+180x^{15}+197x^{16}\\
&+15x^{17}+211x^{18}+99x^{19}+171x^{20}+45x^{21}+54x^{22}+48x^{23}+31x^{24}\\
&+214x^{25}+32x^{26}+115x^{27}+240x^{28}+64x^{29}+186x^{30}+124x^{31}+59x^{32}\\
&+230x^{33}+22x^{34})+y(74+109x+212x^{2}+93x^{3}+208x^{4}+73x^{5}+39x^{6}\\
&+69x^{7}+122x^{8}+158x^{9}+220x^{10}+90x^{11}+125x^{12}+77x^{13}+51x^{14}\\
&+143x^{15}+189x^{16}+83x^{17}+85x^{18}+24x^{19}+233x^{20}+22x^{21}+21x^{22}\\
&+45x^{23}+104x^{24}+229x^{25}+103x^{26}+41x^{27}+76x^{28}+168x^{29}+23x^{30})]/\\
&\mathbf{p}(x),
\end{align*}
\begin{align*}
\mathbf{C}_5(x,y)&=[(226+152x+5x^{2}+62x^{3}+143x^{4}+191x^{5}+255x^{6}+4x^{7}+139x^{8}\\
&+49x^{9}+191x^{10}+239x^{11}+229x^{12}+36x^{13}+198x^{14}+250x^{15}+50x^{16}\\
&+113x^{17}+120x^{18}+2x^{19}+28x^{20}+105x^{21}+193x^{22}+200x^{23}+135x^{24}\\
&+118x^{25}+155x^{26}+76x^{27}+174x^{28}+37x^{29}+204x^{30}+57x^{31}+13x^{32}\\
&+155x^{33}+91x^{34})+y(249+176x+70x^{2}+222x^{3}+121x^{4}+102x^{5}\\
&+207x^{6}+27x^{7}+120x^{8}+240x^{9}+134x^{10}+224x^{11}+152x^{12}+81x^{13}\\
&+162x^{14}+217x^{15}+236x^{16}+103x^{17}+175x^{18}+110x^{19}+19x^{20}+118x^{21}\\
&+228x^{22}+249x^{23}+36x^{24}+183x^{25}+185x^{26}+5x^{27}+16x^{28}+200x^{29}\\
&+95x^{30}+87x^{31})]/\mathbf{p}(x),
\end{align*}
\begin{align*}
\mathbf{C}_6(x,y)&=[(161+113x+91x^{2}+180x^{3}+238x^{4}+167x^{5}+253x^{6}+126x^{7}+156x^{8}\\
&+160x^{9}+216x^{10}+120x^{11}+62x^{12}+174x^{13}+11x^{14}+140x^{15}+147x^{16}\\
&+154x^{17}+6x^{18}+123x^{19}+165x^{20}+110x^{21}+28x^{22}+88x^{23}+237x^{24}\\
&+247x^{25}+226x^{26}+234x^{27}+126x^{28}+256x^{29}+99x^{30}+172x^{31}+75x^{32}\\
&+222x^{33}+34x^{34}+222x^{35})+y(43x^{0}+256x^{1}+48x^{2}+185x^{3}+13x^{4}\\
&+13x^{5}+254x^{6}+151x^{7}+115x^{8}+74x^{9}+163x^{10}+37x^{11}+143x^{12}\\
&+94x^{13}+35x^{14}+110x^{15}+164x^{16}+208x^{17}+241x^{18}+204x^{19}+196x^{20}\\
&+46x^{21}+125x^{22}+55x^{23}+159x^{24}+213x^{25}+68x^{26}+35x^{27}+121x^{28}\\
&+57x^{29}+83x^{30}+196x^{31})]/\mathbf{p}(x)
\end{align*}
with
\begin{align*}
   \mathbf{p}(x)&=135+240x+129x^{2}+163x^{3}+91x^{4}+239x^{5}+229x^{6}+10x^{7}+128x^{8}\\
  &+5x^{9}+13x^{10}+115x^{11}+199x^{12}+154x^{13}+169x^{14}+46x^{15}+144x^{16}\\
  &+187x^{17}+250x^{18}+89x^{19}+195x^{20}+204x^{21}+117x^{22}+149x^{23}\\
  &+166x^{24}+183x^{25}+64x^{26}+66x^{27}+167x^{28}+65x^{29}+24x^{30}+13x^{31}\\
  &+138x^{32}+60x^{33}+x^{34}.
\end{align*}
For example, $P_1=(2,7)$ and $P_2=(121,5)$ are two points on $\mathcal{C}$, and the corresponding effective divisors defining $F(P_1)$ and $F(P_2)$ are cut respectively by
\[x^3 + 239x^2z + 77xz^2 + 90z^3,y^3 + 101y^2z + 61yz^2 + 132z^3\]
and
\[x^3 + 90x^2z + 59xz^2 + 107z^3,y^3 + 59y^2z + 231yz^2 + 192z^3.\]
We checked that $86241[(P_1)-(\infty)]=[(P_2)-(\infty)]$ and that $86241 F(P_1)=F(P_2)$.

\section*{Acknowledgements}

%We thank Jean-Marc Couveignes for giving us the idea about how to construct the eight good functions on the Kummer varirty. We thank Christophe Ritzenthaler, Damien Robert, Nicholas Shepherd-Barron for answering many questions about their papers.
Part of this work has been done while the author worked as a postdoc with Claus Diem at the University of Leipzig. I thank him for suggesting me to make use of Jean-Marc Couveignes and Tony Ezome's approach. I thank Christophe Ritzenthaler for telling me about Milio's paper. I thank Jean-Marc Couveignes for giving me the idea about how to construct the eight good functions on the Kummer varirty. I also thank Nicholas Shepherd-Barron for answering many questions about his paper.

%    Text of article.

%    Bibliographies can be prepared with BibTeX using amsplain,
%    amsalpha, or (for "historical" overviews) natbib style.
\bibliographystyle{amsplain}
%    Insert the bibliography data here.
\bibliography{mybibtex}

\end{document}